\documentclass[12pt]{article}

\usepackage{amsmath,amsthm,amsfonts,amssymb,color,xcolor,subcaption,graphicx} 

\newtheorem{theorem}{Theorem}[section]
\newtheorem{corollary}[theorem]{Corollary}

\newtheorem{assumption}[theorem]{Assumption}

\newtheorem{lemma}[theorem]{Lemma}
\newtheorem{definition}[theorem]{Definition}
\newtheorem{remark}[theorem]{Remark}

\def\cA{\mathcal{A}}
\def\cB{\mathcal{B}}
\def\cC{\mathcal{C}}

\def\cF{\mathcal{F}}

\def\cI{\mathcal{I}}

\def\cL{\mathcal{L}}

\def\cN{\mathcal{N}}
\def\cP{\mathcal{P}}

\def\cS{\mathcal{S}}
\def\cT{\mathcal{T}}

\def\cZ{\mathcal{Z}}

\def\bE{\mathbb{E}}
\def\bQ{\mathbb{Q}}

\def\bP{\mathbb{P}}
\def\bR{\mathbb{R}}

\def\tu{\widetilde{u}}

\begin{document}

\title{Global well-posedness for hyperbolic SPDEs with non-Lipschitz coefficients driven by \\ space-time L\'evy white noise}

\author{Raluca M. Balan\footnote{University of Ottawa, Department of Mathematics and Statistics, 150 Louis Pasteur Private, Ottawa, ON, K1N 6N5, Canada. E-mail address: rbalan@uottawa.ca.} \footnote{Research supported by a grant from Natural Sciences and Engineering Research Council of Canada.}
\and
Juan J. Jim\'enez\footnote{Corresponding author. Department of Mathematics and Statistics, Auburn University, 221 Parker Hall, Auburn, Alabama 36849, USA. E-mail address: juj0003@auburn.edu.}
\and
Llu\'is  Quer-Sardanyons\footnote{Department of
Mathematics, Universitat Aut\`onoma de Barcelona, 08193 Bellaterra (Barcelona),
Spain. E-mail address: lluis.quer@uab.cat.} \footnote{Research supported by grant PID2021-123733NB-I00, Ministerio de Economía y Competitividad, Spain.}
}

\date{December 12, 2025}
\maketitle

\begin{abstract}
\noindent

In this article, we study the global well-posedness of hyperbolic SPDEs on a bounded domain in $\mathbb{R}^d$, driven by a space-time L\'evy white noise, when the drift and diffusion coefficients are
locally Lipschitz and have linear growth. The equations are driven by two types of space-time L\'evy noise: (i) a finite-variance L\'evy white noise; or (ii) a symmetric L\'evy basis that may have infinite variance. A typical example of noise of the second type is the symmetric $\alpha$-stable (S$\alpha$S) random measure with $\alpha \in (0,2)$. 

\end{abstract}


\vspace{1mm}



\section{Introduction}

SPDEs driven by L\'evy noise provide a mathematical framework for modeling diffusive or vibrating systems affected by heavy-tailed randomness. Such models capture discontinuous phenomena found in diverse areas, including financial crashes, turbulent flows, and cosmic radiation. The study of discontinuous random systems dates back to the early 20th century, when classical models like Brownian motion proved insufficient for describing systems with sudden, large jumps. This limitation motivated the development of L\'evy processes, which provide a mathematical framework for modeling discontinuous behavior.

\medskip

There exist several mathematical frameworks for studying SPDEs. The random field approach, introduced by Walsh \cite{walsh86}, interprets solutions as real-valued random fields, while the infinite-dimensional approach, developed by Da Prato and Zabczyk \cite{daprato-zabczyk92}, treats them as stochastic processes taking values in a Hilbert space. We review briefly below the literature related to SPDEs driven by L\'evy noise using the random field approach, which is the topic of the present article. We refer the reader to the monograph \cite{peszat-zabczyk07} for a study of SPDEs with L\'evy noise using the infinite-dimensional approach. 

\medskip

One of the first studies on SPDEs driven by infinite-variance L\'evy noise is \cite{bie98}, which established the existence and uniqueness of the solution to the stochastic heat equation:
\begin{equation}
\label{SHE}
\frac{\partial u}{\partial t}(t,x)=\Delta u(t,x)+ \sigma(u(t,x))\dot{\Lambda}(t,x) \quad t>0,x\in \bR^d
\end{equation}
under the condition
\begin{equation}
\label{bie-cond}
\int_{\bR} |z|^p \nu(dz)<\infty \quad \mbox{for some $0<p<1+\frac{2}{d}$},
\end{equation}
where $\Lambda$ is a L\'evy white noise with L\'evy measure $\nu$ and possibly infinite variance. Condition \eqref{bie-cond} fails for the $\alpha$-stable L\'evy noise, whose L\'evy measure is given by:
\[
\nu(dz)=c_1 (-z)^{-\alpha-1}1_{(-\infty,0)}(z)+c_2 z^{-\alpha-1}1_{(0,\infty)}(z),
\]
for some $c_1,c_2\geq0$ and $\alpha \in (0,2)$. In this case, Mueller \cite{mueller98} established the existence of a solution when $\alpha <1$ and  $\sigma(u)=u^{\gamma}$, while Mytnik \cite{mytnik02} addressed the case $\alpha >1$.

The existence of solution for a general SPDE driven by an $\alpha$-stable L\'evy noise with Lipschitz coefficient $\sigma$ was established in \cite{balan14} for equations on bounded domains. For equations on the entire space $\bR^d$, a major advance occurred in article \cite{chong17-SPA}, which established the existence of a random field solution to equation \eqref{SHE}, under the assumption that:
\[
\int_{|z|\leq 1}|z|^p \nu(dz)<\infty \quad \mbox{and} \quad \int_{|z|>1} |z|^q\nu(dz)<\infty,
\]
for some $0<q\leq p$. In \cite{CDH19}, the authors proved the regularity of the sample paths of the solution of the stochastic heat equation on bounded domains, and on the entire space $\bR^d$. The uniqueness of the solution was proved in \cite{BCL23} for equation \eqref{SHE} with $\sigma(u)=\beta u$ for some $\beta>0$. In the case of the stochastic wave equation driven by infinite variance L\'evy noise, on the entire domain $\bR^d$ (in dimension $d\leq 2$), with globally Lipschitz coefficient $\sigma$, the existence of the solution was proved in \cite{balan23}, while the uniqueness was shown in \cite{jimenez24}.

In \cite{LOP04}, the authors developed a white noise framework for the study of SPDEs driven by a $d$-parameter (pure jump) L\'evy white noise, and illustrated this theory for solving the stochastic Poisson equation with L\'evy white noise.

\medskip

In the recent years, there has been an increased interest in studying SPDEs with locally Lipschitz coefficients. Many advances have been made in the case of equations driven by Gaussian noise. These equations may not have a global solution. Moreover, a global solution may not be unique, as shown for instance in \cite{marinelli-rockner10,MMP14,mytnik-perkins11,MPS06}.

In \cite{DKZ19}, the authors studied the stochastic heat equation on $(0,1)$, driven by the space-time Gaussian white noise $W$:
\begin{equation}
\label{SHE1}
\frac{\partial u}{\partial t}(t,x)=\frac{\partial^2 u}{\partial x^2}(t,x)+b\big(u(t,x)\big)
+\sigma\big( u(t,x)\big) \dot{W}(t,x), \quad t>0,x\in (0,1)
\end{equation}
with Dirichlet boundary conditions $u(t,0)=u(t,1)=0$, and proved that if 
\[
b(z)=O(|u|\log|z|) \quad \mbox{and} \quad \sigma(z)=O(|z|\log |z|),
\] 
then there exists a global solution. Articles \cite{bonder-groisman09,FN21} show that, if $\sigma$ is constant and $b$ is locally Lipschitz, non-decreasing on $(0,\infty)$ and non-negative, then the solution of \eqref{SHE1} blows up in finite time if and only if {\em the finite Osgood-condition} holds:
\begin{equation}
\label{osgood}
\int_{a}^{\infty}\frac{1}{b(z)}dz<\infty \quad \mbox{for some $a>0$}.
\end{equation}
In \cite{FKN24-1}, the authors considered the same problem for the stochastic heat equation on the entire domain $\bR^d$, and proved that if $b$ is locally Lipschitz, non-negative and non-decreasing, and $\sigma$ is globally Lipschitz with $0<\inf_{z\in \bR} \sigma(z)\leq \sup_{z\in \bR} \sigma(z)<\infty$, then
condition \eqref{osgood} implies that the solution {\em blows up everywhere and instantaneously}, i.e.
\[
\bP\big(u(t,x)=\infty \quad \mbox{for all $t>0$ and $x \in \bR$}\big)=1,
\]
where $u(t,x)=\lim_{N\to \infty} u_N(t,x)$ and $u_N$ is solution of the equation with drift coefficient $b_N(z)=b(z \wedge N)$. If $\sigma(z)=z$, it is expected that the solution will still blow up, as observed in \cite{joseph-ovhal25} in the discrete case of interacting SDEs.  Recent extensions have been provided in \cite{CFHS} for the colored noise in time: in particular, their results show that if $b(z)=z^p$ for some $p>1+\frac{2}{d}$, then the solution will not explode, provided that the initial condition is small enough.

Special attention has been given to the case of the stochastic heat equation with super-linear locally Lipschitz coefficients, driven by space-time Gaussian white noise. The case $d=1$ was considered in article \cite{Salins-EJP}, in which the author proved the existence of a global solution when $b$ satisfies the {\it infinite Osgood-condition} and  $\sigma$ grows super-linearly. 
In higher dimensions, the existence of a global solution on the entire domain $\mathbb{R}^d$ was proved in \cite{Salins-SD2022}, when $b$ satisfies an Osgood-type condition and $\sigma$ is constant. The extension of this result to the case when $\sigma$ may have super-linear has been addressed in \cite{CFHS} (see also \cite{CH-ProcAMS}). Finally, in \cite{Salins-TransAMS} a general parabolic SPDE on a bounded domain of $\mathbb{R}^d$ has been considered, assuming that the coefficients $b$ and $\sigma$ grow superlinearly, and $b$ may not be dissipative. In the same article, the author provides some Osgood-type conditions on $b$ and $\sigma$ under which there is no explosion of the solution in finite time. 

\medskip

In the present article, we continue this line of investigation dedicated to SPDEs with locally Lipschitz coefficients. The novelty lies in establishing the well-posedness of random-field solutions for a hyperbolic class of SPDEs driven by a space-time L\'evy white noise with possibly infinite variance, such as {\em the symmetric $\alpha$-stable L\'evy white noise} for $\alpha \in (0,2)$ in space-time.This is the first work in the literature to treat non-Lipschitz coefficients using the random field approach, for SPDEs driven by a general class of L\'evy white noises, including those with infinite variance.

The problem of existence of solutions for stochastic differential equations (SDEs) with L\'evy noise and locally Lipschitz coefficients was studied in the literature,
but not as extensively as for SDEs with Gaussian noise; see for example \cite{XPG15} and the references therein. Stochastic evolution equations with locally Lipschitz coefficients have also been studied in different frameworks. The recent article \cite{KM25} examines the stochastic heat equation with non-Lipschitz coefficients, driven by a time-homogeneous compensated Poisson random measure. To the best of our knowledge, the only work that studies SPDEs driven by a space-time L\'evy white noise in a context similar to ours is article \cite{WYZ24}, which proves the existence of a weak solution for the stochastic heat equation with non-Lipschitz coefficients, driven by a truncated $\alpha$-stable white noise. This is achieved using an approximation weak convergence method for the coefficients. However, it is not clear if the framework in \cite{WYZ24} is equivalent to ours, and if this weak solution satisfies also the mild formulation. 

\medskip

We now introduce the framework for this article. We will consider an equation on a bounded domain $D \subset \bR^d$  whose {\em boundary $\partial D$ is of class $C^{\infty}$}, i.e. $\partial D$ is a smooth $(d-1)$-dimensional manifold. (The fact that $D$ has a smooth boundary is required for Theorem \ref{embed-th} below, which is a key result for this article.) The equation will contain the fractional (or spectral) power $(-\Delta)^{\gamma}$ of the Laplacian, for $\gamma>0$, which is defined by:
\[
(-\Delta)^{\gamma} e_k=\lambda_k^{\gamma} e_k,
\]
where $0<\lambda_1\leq \lambda_2 \leq \ldots$ are the eigenvalues of $-\Delta$ with Dirichlet boundary conditions, and $(e_k)_{k\geq 1}$ are the corresponding normalized eigenfunctions forming a complete orthonormal basis of $L^2(D)$. Additionally, throughout this work, we denote by $H_r(D)$ {\em the fractional Sobolev space of order $r \in \bR$}; see Appendix \ref{app-Sob}. 

\medskip

For the mild formulation of the solution, we will use the Green function $G_t(x,y)$ of the operator $\dfrac{\partial^2 }{\partial t^2} +(-\Delta)^{\gamma}$ on $D$ with Dirichlet boundary conditions. The function $G_t$ has the following spectral representation:
\begin{equation}
\label{def-G}
G_t(x,y)=\sum_{k\geq 1}\frac{\sin(\lambda_k^{\gamma/2} t)}{\lambda_k^{\gamma/2}} e_k(x)e_k(y)1_{\{t\geq 0\}}  \quad  x,y\in D.
\end{equation}
Representation \eqref{def-G} can be deduced following the same arguments on pages 183-184 in \cite{duffy}.
An important property for our analysis will be the fact that, if $\gamma>d$, then 
\begin{equation}
    \label{square-G}
    \int_0^T \left(\sup_{x \in D} \int_D G_t^2 (x,y) dy \right) dt < \infty \quad \mbox{for any $T>0$}.
\end{equation}
For the sake of completeness, we prove this result in Lemma \ref{lem-square-G}.

\medskip

As mentioned above, we will consider two driving L\'evy noise processes. We include below the definitions of these processes, which are based on the same Poisson random  measure $J$. 

\medskip

Let $J$ be a Poisson random measure on $\bR_{+} \times \bR^d \times \bR_{0}$, defined on a complete probability space $(\Omega,\cF,\bP)$, of intensity $\mu(dt,dx,dz)=dtdx \nu(dz)$, where $\nu$  is a {\em L\'evy measure} on $\bR_0$, that is $\nu$ satisfies the condition:
\begin{equation}
\label{Levy-measure}
\int_{\bR_0}(z^2 \wedge 1)\nu(dz)<\infty.
\end{equation}
Here $\bR_{0}:=\bR \verb2\2 \{0\}$ is equipped with the distance $d(x,y)=|x^{-1}-y^{-1}|$, so that bounded sets in $\bR_0$ are in fact bounded away from 0. We let $\widetilde{J}(F)=J(F) -\mu(F)$ be the compensated version of $J$, for any bounded Borel set $F$ in $\bR_{+} \times \bR^d \times \bR_{0}$.

\bigskip

In the first part of this article (Section \ref{section-fin-var}), we consider a (finite-variance) {\em L\'evy white noise} $L=\{L(B);B \in \cB_{b}(\bR_{+}\times \bR^d)\}$, defined by:
\begin{equation}
\label{def-L1}
L(B)=\int_{B \times \bR_0}z \widetilde{J}(dt,dx,dz), 
\end{equation}
where $\cB_b(\bR_+ \times \bR^d)$ is the class of bounded Borel subsets of $\bR_{+}\times \bR^d$, and $\nu$ satisfies:
\begin{equation}
\label{def-m2}
m_2:=\int_{\bR_0}z^2 \nu(dz)<\infty.
\end{equation}

We refer the reader to Section \ref{subsect-prelim} for the definition of the stochastic integral with respect to $L$, which uses the concept of predictability.

\medskip

We consider the fractional stochastic wave equation on $D$, driven by noise $L$:
\begin{align}
\label{fSWE}
	\begin{cases}
		\dfrac{\partial^2 u}{\partial t^2} (t,x)
		=  -(-\Delta)^{\gamma}u(t,x) + b\big(u(t,x)\big) + \sigma\big(u(t,x)\big) \dot{L}(t,x), \
		t>0, \ x \in D, \\
		u(0,x) = 0, \quad \dfrac{\partial u}{\partial t}(0,x)=0, \quad x \in D, \\
u(t,x)=0, \quad  t > 0, \ x \in \partial D.
	\end{cases}
\end{align}

We have the following definition.

\begin{definition}
{\rm 
A process $\{u(t,x);t\geq 0,x\in D\}$ is a \underline{local solution} of equation \eqref{fSWE} if it is predictable and there exists a stopping time $\tau$ such that, for any $t\geq 0$ and $x\in D$,
\[
u(t,x)=\int_0^t \int_{D}G_{t-s}(x,y) b\big(u(s,y)\big)dyds+  \int_0^t \int_{D}
G_{t-s}(x,y)\sigma\big(u(s,y)\big)L(ds,dy),
\]
almost surely (a.s.) on the event $\{t<\tau\}$, where the stochastic integral is the It\^o integral. If we can take $\tau=\infty$, then $u$ is a \underline{global solution}.
}
\end{definition}

In view of Lemma \ref{fact1-lem}, the fact that $u$ is a local solution means that there exists a stopping time $\tau$ such that, for any $t\geq 0$ and $x \in D$, 
\begin{align}
1_{\{t<\tau\}} u(t,x)& =1_{\{t<\tau\}} \int_0^t \int_{D}G_{t-s}(x,y) b\big(u(s,y)\big)dyds\\
\label{def-local-sol}
& \quad \quad + \ 1_{\{t<\tau\}} \int_0^t \int_{D}
G_{t-s}(x,y)\sigma\big(u(s,y)\big) L(ds,dy) \quad \mbox{a.s.}
\end{align}

We introduce the following assumption:

\begin{assumption}
\label{assumpt}
{\rm 
 (i) $b$ and $\sigma$ are {\em locally Lipschitz}, i.e. for any $N>0$, there exist constants $C_{b,N},C_{\sigma,N}>0$ such that for any $\xi,\eta\in [-N,N]$,
\[
|b(\xi)-b(\eta)|\leq C_{b,N}|\xi-\eta| \quad \mbox{and} \quad |\sigma(\xi)-\sigma(\eta)|\leq C_{\sigma,N}|\xi-\eta|,
\]
(ii) $b$ and $\sigma$ {\em have linear growth}, i.e. there exist constants $D_{b},D_{\sigma}>0$ such that for any $\xi \in \bR$,
\begin{equation}
\label{lin-grow}
|b(\xi)|\leq D_{b}(1+|\xi|) \quad \mbox{and} \quad |\sigma(\xi)|\leq D_{\sigma}(1+|\xi|).
\end{equation}
}
\end{assumption}

We recall that a function $f:\bR_{+}\to \bR$ is called ``c\`adl\`ag'' if it is right-continuous and has left limits. We are now ready to state the first main result of this article.

\begin{theorem}
\label{main}
Let $L$ be the L\'evy white noise given by \eqref{def-L1}. Assume that $\nu$  satisfies \eqref{def-m2}.
If $\gamma>d$ and Assumption \ref{assumpt} holds, then for any $r \in (\frac{d}{2}, \gamma - \frac{d}{2})$ fixed, equation \eqref{fSWE} has a unique
global solution $u$ which satisfies
\begin{equation}
\label{u-cadlag}
\bP\big(u(t,\cdot) \in H_r(D) \ \mbox{for any $t\geq 0$ and $t \mapsto \|u(t, \cdot) \|_{H_r(D)}$ is c\`adl\`ag}\big)=1.
\end{equation}
\end{theorem}

We explain briefly the strategy for constructing the global solution to equation \eqref{fSWE}. In the case of equations with locally Lipschitz coefficients, driven by space-time Gaussian white noise, a common procedure for equations on bounded domains (used for instance in \cite{DKZ19,FN21,CFHS}) is to consider the stopping time:
\[
T_N=\inf\{t> 0; \sup_{x\in D}|u_N(t,x)|>N\}, \quad (\inf \emptyset=\infty)
\]
where $u_N$ is the solution of the equation with truncated coefficients $(b_N,\sigma_N)$, and identify sufficient conditions on $(b,\sigma)$ under which $\lim_{N \to \infty} T_N = \infty$ almost surely. 
This method uses the fact that the solution $u_N$ is continuous in $(t,x)$, which implies that
$u_N(t,\cdot) \in L^{\infty}(D)$ for all $t\geq 0$, and hence $T_N$ is well-defined.
In the recent preprint \cite{FKN24}, a different method was used, which shows the convergence of the sequence $\{u_N(t,x)\}_{N\geq 1}$ in $L^2(\Omega)$ uniformly in $(t,x)$, without using the stopping time $T_N$. This method relies on proving some exponential bounds on the tail probabilities of $u_N(t,x)$, and a careful analysis on the dependence on $N$ in these bounds. 

In the case of equations driven by L\'evy noise, the solution may not be bounded (nor continuous) on $D$, and hence $T_N$ is not well-defined. To circumvent this difficulty, we prove that $u_N$ has a modification $\tu_{N}$ with values in the Sobolev space $H_r(D)$ with $0<r<\gamma-\frac{d}{2}$, for which the map $t \mapsto \|\tu_N(t,\cdot)\|_{H_r(D)}$ is c\`adl\`ag. Then, we consider the stopping time:
\[
\tau_N=\inf\{t> 0; \|\tu_N(t,\cdot)\|_{H_r(D)} >N\} \quad (\inf \emptyset=\infty).
\]
Assuming in addition that $r>d/2$ (which forces the restriction $\gamma>d$), we use the embedding of $H_r(D)$ into $L^{\infty}(D)$, and we provide an exponential bound on the second moment of $\sup_{t\leq T} \|\tu_N(t,\cdot)\|_{H_r(D)}$ which does not depend on $N$ (see Theorem \ref{tu-th2} below). Finally, as in \cite{DKZ19}, we show that with probability 1:
(i) $u_N(t,x)=u_{N+1}(t,x)$ if $t<\tau_N$; 
(ii) $\lim_{N \to \infty}\tau_N=\infty$.
This procedure allows us to construct a global solution by pasting the solutions $(u_N)_{N\geq 1}$, i.e. defining $u(t,x)=u_{N}(t,x)$ for $\tau_{N-1} \leq t <\tau_N$, with $\tau_0=0$. Property (i) ensures that $u(t,x)= u_{N}(t,x)$ if $t<\tau_N$, while (ii) shows that $u$ is  global.

\bigskip

In the second part of this paper (Section \ref{section-inf}), we assume that the noise is given by a (pure-jump, no drift) {\em L\'evy basis} $\Lambda=\{\Lambda(A);A \in \cP_b\}$, defined by: 
\begin{equation}
\label{def-L2}
\Lambda(A)=\int_0^{\infty}\int_{\bR^d} \int_{\{|z|\leq 1\}} 1_{A}(t,x) z \widetilde{J}(dt,dx,dz)+
\int_0^{\infty}\int_{\bR^d} \int_{\{|z|> 1\}} 1_{A}(t,x) z J(dt,dx,dz),
\end{equation}
where $J$ and $\widetilde{J}$ are as above, and we assume that $\nu$ is {\em symmetric}. $\nu$ may not satisfy \eqref{def-m2}.
Here $\cP_b$ is the class of bounded predictable sets, given by Definition \ref{def-pred} below.
We refer the reader to Section \ref{subsect-Levy-bases} for the definition and properties of the stochastic integral with respect to $\Lambda$.

\begin{remark}
{\rm
(i) The process $\{Z(B)=\Lambda(\Omega \times B);B \in \cB_b(\bR_{+}\times \bR^d)\}$ is an ID independently scattered random measure, as defined in \cite{RR89}. $Z(B)$ may not have any moments. 

(ii)
If the measure $\nu$ is given by:
\[
\nu(dz)=\frac{1}{2}\alpha|z|^{-\alpha-1}dz \quad \mbox{for some} \quad \alpha \in (0,2),
\]
we say that $\Lambda$ is a {\em symmetric $\alpha$-stable (S$\alpha$S) L\'evy basis}.
The process $\{Z(B);B \in \cB_b(\bR_{+}\times \bR^d)\}$ is a S$\alpha$S random measure, as defined in \cite{ST94}.
}
\end{remark}
 
We consider the fractional stochastic wave equation on $D$, with noise $\Lambda$: 
\begin{align}
\label{fSWE2}
	\begin{cases}
		\dfrac{\partial^2 u}{\partial t^2} (t,x)
		=  -(-\Delta)^{\gamma}u(t,x) + b\big(u(t,x)\big) + \sigma\big(u(t,x)\big) \dot{\Lambda}(t,x), \
		t>0, \ x \in D, \\
		u(0,x) = 0, \quad \dfrac{\partial u}{\partial t}(0,x)=0, \quad x \in D, \\
u(t,x)=0, \quad  t > 0, \ x \in \partial D.
	\end{cases}
\end{align}

An analogue of \eqref{fSWE2} on $\bR^d$ was studied in \cite{DS05}, in the Gaussian case.

\begin{definition}
{\rm 
A process $\{u(t,x);t\geq 0,x\in D\}$ is a \underline{local solution} of equation \eqref{fSWE2} if it is predictable and there exists a stopping time $\tau$ such that, for any $t\geq 0$ and $x\in D$,
\[
u(t,x)=\int_0^t \int_{D}G_{t-s}(x,y) b\big(u(s,y)\big)dyds+  \int_0^t \int_{D}
G_{t-s}(x,y)\sigma\big(u(s,y)\big)\Lambda(ds,dy),
\]
almost surely (a.s.) on the event $\{t<\tau\}$, where the stochastic integral is interpreted in the sense of Definition \ref{def-integrable}. If we can take $\tau=\infty$, then $u$ is a \underline{global solution}.
}
\end{definition}

We are now ready to state the second main result of this article.

\begin{theorem}
\label{main2}
Let $\Lambda$ be a L\'evy basis given by \eqref{def-L2}. Assume that $\nu$ is symmetric and satisfies \eqref{Levy-measure}.
If $\gamma>d$, and Assumption \ref{assumpt} holds, then for any $r \in (\frac{d}{2}, \gamma - \frac{d}{2})$ fixed, equation \eqref{fSWE2} has a unique  global solution $u$ 
which satisfies \eqref{u-cadlag}.
\end{theorem}

\begin{corollary}
Let  $\Lambda$ be a S$\alpha$S L\'evy basis with $\alpha \in (0,2)$.
If $\gamma>d$ and Assumption \ref{assumpt} holds, then for any $r \in (\frac{d}{2}, \gamma - \frac{d}{2})$ fixed, equation \eqref{fSWE2} has a unique global solution $u$ which satisfies \eqref{u-cadlag}.
\end{corollary}

\begin{remark}
{\rm
Due to the restriction $\gamma>d$, 
the methods presented in this article do not apply to the stochastic wave equation (for which $\gamma=1$).
}
\end{remark}

\begin{remark}
	{\rm Article \cite{CDH19} studied the stochastic heat equation on $(0,T) \times D$, driven by a (pure-jump) L\'evy basis, with drift coefficient $b=0$, globally Lipschitz diffusion coefficient $\sigma$, 
and Dirichlet boundary conditions, where $D$ is a $C^{\infty}$-regular domain in $\bR^d$. By Proposition 2.17 {\em ibid.}, this equation has a global solution, provided that
\[
\int_{|z|\leq 1}|z|^p \nu(dz)<\infty \quad \mbox{for some $0<p<1+\frac{2}{d}$.}
\]
Theorem 2.19 {\em ibid.} shows that this solution has a c\`adl\`ag modification with values in $H_r(D)$, for any $r < -\tfrac{d}{2}$. This constraint on $r$ is optimal. To see this, it suffices to consider the basic case of a compound-Poisson L\'evy white noise in dimension $d=1$; see Remark 2.6 in \cite{CDH19}. Consequently, the techniques developed in this present work cannot be applied to the parabolic analogue of \eqref{fSWE}, due to this optimal constraint and Theorem \ref{embed-th}.
	} 
\end{remark}

To prove Theorem \ref{main2}, we will use a strategy that has been used in the literature for SPDEs driven by heavy tailed L\'evy noise. We first consider the truncated noise $\Lambda^K$, which is obtained by removing the ``large'' jumps (which exceed a fixed value $K$). Due to the symmetry assumption on $\nu$, $\Lambda^K$ becomes a finite variance L\'evy noise, as \eqref{def-L1}. By Theorem \ref{main}, equation \eqref{fSWE} with noise $L$ replaced by $\Lambda^K$ has a global solution $u^K$. We show that the solutions $(u^K)_{K\geq 1}$ are consistent, in the sense that $u^K(t,x)=u^{K+1}(t,x)$ if $t<\tau^K$, where $\tau^K$ is the first time when $J$ has a jump that exceeds $K$. Finally, by pasting the solutions $(u^K)_{K\geq 1}$, we obtain a global solution for equation \eqref{fSWE}.

\bigskip

The rest of the article is organized in two sections, which are dedicated to the proofs of Theorem \ref{main} (Section \ref{section-fin-var}), respectively Theorem \ref{main2} (Section \ref{section-inf}). The appendix contains some auxiliary results.

\section{The finite variance case}
\label{section-fin-var}

In this section, we consider the equation driven by the finite variance L\'evy noise $L$. In Section \ref{subsect-prelim}, we introduce some preliminaries about It\^o integration with respect to $L$. In Section \ref{subsect-glob}, we examine the equation with noise $L$ and globally Lipschitz coefficients $(b,\sigma)$. Finally, in Section \ref{section-loc-Lip}, we present the proof of Theorem \ref{main}.

\subsection{Preliminaries}
\label{subsect-prelim}

In this section, we include some background material which is used in this paper.

\medskip

We start by recalling that the filtration induced by $J$ is given by:
\begin{equation}
\label{def-filtr}
\cF_t=\sigma\big\{J([0,s] \times A \times B);s \in [0,t],A \in \cB_b(\bR^d),B \in \cB_b(\bR_0)\big\} \vee \cN,
\end{equation}
where $\cN$ is the class of $\bP$-negligible sets, $\cB_b(\bR^d)$ is the class of bounded Borel sets in $\bR^d$, and $\cB_b(\bR_0)$ is the class of bounded Borel subsets of $\bR_0$.
The filtration $(\cF_t)_{t\geq 0}$ is right-continuous, i.e.
\[
\cF_t=\cF_{t+} \ \mbox{for all $t\geq 0$}, \quad \mbox{where} \quad \cF_{t+}:=\bigcap_{s\leq t} \cF_s.
\]
All stopping times and martingales considered in this article are with respect to 
 $(\cF_t)_{t\geq 0}$.

\medskip

We now introduce the concepts of predictable sets and predictable  processes. We start by recalling that an {\em elementary process} is a linear combination of processes of the form
\[
X(t,x)=Y 1_{(a,b]}(t) 1_{A}(x), \quad t\geq 0,x \in \bR^d,
\]
where $0\leq a<b$, $Y$ is $\cF_a$-measurable and $A \in \cB_d(\bR^d)$. 

\begin{definition}
\label{def-pred}
{\rm
The {\em predictable $\sigma$-field} $\cP$ on $\Omega \times \bR_{+} \times \bR^d$ is the $\sigma$-field generated by elementary processes. 
A set $A \in \cP$ is called {\em predictable}.
We  let $\cP_b$ be the class of all bounded predictable sets $A$, i.e. sets $A \in \cP$ such that $A \subseteq \Omega \times [0,k] \times [-k,k]^d$ for some $k>0$. 
A process $\{X(t,x);t\geq 0,x\in \bR^d\}$ is called {\em predictable} if it is $\cP$-measurable.}
\end{definition}

Recall that the noise $L$ is given by \eqref{def-L1}.
Then, for any set $B \in \cB_b(\bR_+ \times \bR^d)$,
\[
\bE[L(B)]=0 \quad \mbox{and} \quad \bE[|L(B)|^2]=m_2 |B|,
\]
where $|B|$ is the Lebesgue measure of $B$. Therefore, we can define the {\em It\^o integral} $\int H dL$ with respect to $L$, for any predictable process $H$ such that
$
\bE \left[\int_0^T \int_{\bR^d} |H(t,x)|^2 dxdt\right]<\infty$ for all $T>0$. This integral is an isometry:
\begin{equation}
\label{isometry}
\bE \left[\left|\int_0^T \int_{\bR^d}H(t,x)L(dt,dx)\right|^2\right]=m_2 \bE \left[\int_0^T \int_{\bR^d}|H(t,x)|^2 dxdt\right],
\end{equation}
and the process $\{\int_0^t \int_{\bR^d} H(s,x)L(ds,dx);t\geq 0\}$ is a square-integrable martingale.

\subsection{The globally Lipschitz case}
\label{subsect-glob}

Because our method is based on truncating the coefficients $b$ and $\sigma$, we consider first the equation with globally Lipschitz coefficients. Therefore,
in this sub-section, we assume that $b$ and $\sigma$ are globally Lipschitz, with Lipschitz constants $L_b$ and $L_{\sigma}$: for any $\xi,\eta\in \bR$,
\[
|b(\xi)-b(\eta) |\leq L_{b}|\xi-\eta| \quad \mbox{and} \quad |\sigma(\xi)-\sigma(\eta)|\leq L_{\sigma}|\xi-\eta|.
\]
Clearly, this implies that $b$ and $\sigma$ satisfy the linear growth condition \eqref{lin-grow} with constants $D_b=\max\{L_b,b(0)\}$ and $D_{\sigma}=\max\{L_{\sigma},\sigma(0)\}$.


\begin{theorem}
\label{exist-th}
Let $L$ be the L\'evy white noise given by \eqref{def-L1}.
Assume that $b$ and $\sigma$ are globally Lipschitz and 
$\gamma>d$. Then, equation \eqref{fSWE} has a unique global solution $u$  satisfying 
\begin{equation}
	\label{def-KT}
	K_T:=\sup_{t \in [0,T]}\sup_{x \in D}\bE[|u(t,x)|^2]<\infty \quad \mbox{for all $T>0$}.
\end{equation}
More precisely, for any $t\geq 0$ and $x \in D$, 
\[
u(t,x)=\int_0^t \int_{D} G_{t-s}(x,y)b\big( u(s,y)\big)dyds+\int_0^t \int_{D}
G_{t-s}(x,y)\sigma\big( u(s,y)\big)L(ds,dy) \quad \mbox{a.s.}
\]
Consequently, $\int_D \bE [|u(t,x)|^2]dx<\infty$, and hence, $u(t,\cdot) \in L^2(D)$ a.s. for any $t>0$.
\end{theorem}

\begin{proof}
This follows by classical methods (e.g. the proof of Theorem 13 of \cite{dalang99}), using Picard's iterations and relation \eqref{square-G}.
\end{proof}

We recall the following definition.

\begin{definition}
{\rm  We say that $\tu$ is a {\em modification} of $u$ if
\begin{equation}
\label{u-tu}
\bP\big(u(t,x)=\tu(t,x) \big)=1 \quad \mbox{for any $t \geq 0$ and $x \in D$}.
\end{equation}
}
\end{definition}

\begin{remark}
{\rm If $u$ is a global solution of \eqref{fSWE} and $\tu$ is a modification of $u$, then  
$\tu$ is also a global solution of \eqref{fSWE}.}
\end{remark}

We recall that $0<\lambda_1\leq \lambda_2 \leq \ldots$ denote the eigenvalues of $-\Delta$ with Dirichlet boundary conditions, and $(e_k)_{k\geq 1}$ are the corresponding normalized eigenfunctions.
We will use the following asymptotic property of $(\lambda_k)_{k\geq 1}$, called the {\em Weyl's Law} (see Theorem 1.11 of \cite{Birman}):
\begin{equation}
\label{weyl}
\lim_{k \to \infty}\frac{\lambda_k}{k^{2/d}}=\frac{4\pi \Gamma(1+\frac{2}{d})^{2/d}}{|D|^{2/d}}.
\end{equation}
Hence, there exist $c_1,c_2>0$ such that
$c_1 k^{2/d} \leq \lambda_k \leq c_2 k^{2/d}$ for all $k\geq 1$, and so,
\begin{equation}
\label{sum-lam}
\sum_{k \geq 1}\lambda_k^{r-\gamma}<\infty \quad \mbox{if and only if} \quad r<\gamma-\frac{d}{2}.
\end{equation}
Moreover, by relation (A.2) of \cite{CD23},
\begin{equation}
\label{en-bound}
\|e_k\|_{L^{\infty}(D)} \leq \left(\frac{2e \lambda_k}{\pi d}\right)^{d/4}.
\end{equation}
Weyl's law and \eqref{en-bound} are used to prove property \eqref{square-G} of the Green function $G$.

\begin{lemma}
\label{lem-square-G}
If $\gamma > d$, then relation \eqref{square-G} holds.
\end{lemma}

\begin{proof} 
Using the expansion \eqref{def-G} and the fact that $(e_k)_{k\geq 1}$ form an orthonormal basis of $L^2(D)$, we have
\[
 \int_D G^2_t(x,y)dy=\sum_{k\geq 1} \frac{\sin^2(\lambda_k^{\gamma/2} t)}{\lambda_k^\gamma} e^2_k(x) \
  \leq 
 \sum_{k\geq 1} \lambda_k^{-\gamma} e^2_k(x).
\]
Then, 
\begin{equation*}
 \int_0^T\left(\sup_{x\in D} \int_D G^2_t(x,y)dy\right)dt  \leq  T \sup_{x\in D}
 \sum_{k \geq 1} \lambda_k^{- \gamma} e^2_k(x) \leq T \sum_{k\geq 1} \lambda_k^{-\gamma} \|e_k\|_{L^{\infty}(D)}^2.
\end{equation*}
By \eqref{en-bound} and relation \eqref{sum-lam} with $r=d/2$, the last series converges, provided that $\gamma > d$.
\end{proof}

\begin{remark}
{\rm Under additional conditions, it is possible to show that relation
\eqref{square-G} holds for any $\gamma>d/2$. One such condition is the fact that $D$ satisfies {\em the cone property}, in which case, we can argue similarly to the proof of (A.7) of \cite{CD23} (using Theorem 8.2 of \cite{agmon}).
}
\end{remark}

For the next result, we refer the reader to Appendix \ref{app-Sob} for the definition of the Sobolev space $H_r(D)$. We denote by $|D|$ the Lebesque measure of $D$.

\begin{theorem}
\label{tu-th}
Let $L$ be the L\'evy white noise given by \eqref{def-L1}.
Assume that $b$ and $\sigma$ are globally Lipschitz and $\gamma>d$. Let $u$ be the solution of equation \eqref{fSWE}. Let $\frac{d}{2}<r<\gamma-\frac{d}{2}$ be arbitrary. Then, there exists a modification $\tu$ of $u$ such that, for any $T>0$,
\begin{equation}
\label{sup-Hr}
\bE\big[\sup_{t \in [0,T]} \|\tu(t,\cdot)\|_{H_r(D)}^2\big]\leq (T+1)T \widehat{C}_T (1+ K_T),
\end{equation}
where $K_T$ is given by \eqref{def-KT}, and 
\begin{equation}
\label{def-hatC}
\widehat{C}_T:=4\big(T |D| D_{b}^2+16 m_2 D_{\sigma}^2 \big)\sum_{k\geq 1}\lambda_k^{r-\gamma}.
\end{equation}
Consequently, $\bP\big(\tu(t,\cdot) \in H_r(D) \ \mbox{for any $t\geq 0$}\big)=1$. Moreover, the function $t \mapsto \|\tu(t,\cdot)\|_{H_r(D)}$ is c\`adl\`ag on $[0,\infty)$, with probability 1.
\end{theorem}

\begin{proof}
{\em Step 1.} In this step we prove that for any $T>0$ fixed, there exists a process
$\{\tu(t,x);t\in [0,T],x\in D\}$ such that \eqref{u-tu} holds for any $t\in [0,T]$ and $x\in D$, and for which \eqref{sup-Hr} is satisfied. In {\em Step 2}, we will show that this process can be extended to $[0,\infty)$.  

\medskip

We define
\begin{align*}
I(t,x)&=\int_0^t \int_{D} G_{t-s}(x,y) b\big(u(s,y)\big) dyds \\
Z(t,x)&=\int_0^t \int_{D} G_{t-s}(x,y)\sigma\big(u(s,y)\big) L(ds,dy).
\end{align*}
By Theorem \ref{exist-th}, for any $t\geq 0$ and $x \in D$, we have:
\begin{equation}
\label{eq-tx}
u(t,x)=I(t,x)+Z(t,x) \quad \mbox{a.s.}
\end{equation}
where the negligible set depends on $(t,x)$.

We study $I$ first. The $k$-th Fourier coefficient of $I(t,\cdot)$ is:
\begin{align*}
I_k(t):=\cF_k [I(t,\cdot)]
=\langle I(t,\cdot),e_k\rangle_{L^2(D)} 
&=\int_0^t \int_D \left(\int_D G_{t-s}(x,y)e_k(x) dx\right) b\big(u(s,y)\big)dyds\\
&
=\frac{1}{\lambda_k^{\gamma/2}} \int_0^t \int_{D} \sin\big(\lambda_k^{\gamma/2}(t-s)\big)e_k(y) b\big(u(s,y)\big)dyds,
\end{align*}
where for the last line we used the fact that:
\begin{equation}
\label{Fourier-G}
\cF_k [G_t(\cdot,y)]=\frac{\sin\big(\lambda_k^{\gamma/2}t\big)}{\lambda_k^{\gamma/2}}e_k(y).
\end{equation}
By the Cauchy-Schwarz inequality, and the linear growth property of $b$, 
\begin{align}
\nonumber
\sup_{t\leq T }\big|I_k(t)\big|^2 & \leq   \frac{T|D|}{\lambda_k^{\gamma}} \int_0^T \int_{D} \sin^2\big(\lambda_k^{\gamma/2}(t-s)\big) |e_k(y)|^2 |b\big(u(s,y)\big)|^2dyds \\
\label{bound-FI}
& \leq  \frac{2T|D| D_b^2}{\lambda_k^{\gamma}} \int_0^T \int_{D} |e_k(y)|^2 \big(1+|u(s,y|^2\big) dyds
\end{align}
and hence, since $\int_{D}|e_k(y)|^2dy=1$,
\begin{align}
\label{bound-I2-0}
\bE\big[\sup_{t\leq T}\big|I_k(t)\big|^2\big] 
& \leq   \frac{2T|D| D_b^2}{\lambda_k^{\gamma}}  \int_0^T \int_{D} |e_k(y)|^2 \big(1+\bE[|u(s,y)|^2]\big) dyds\\
\nonumber
& \leq \frac{2T|D| D_b^2}{\lambda_k^{\gamma}} \big(1+\sup_{t \in [0,T]}\sup_{x\in \bR}\bE[|u(t,x)|^2]\big) \int_0^T \int_{D} |e_k(y)|^2 dyds\\
\nonumber
& =\frac{2T^2|D| D_b^2}{\lambda_k^{\gamma}} \big(1+\sup_{t \in [0,T]}\sup_{x\in \bR}\bE[|u(t,x)|^2]\big)<\infty.
\end{align}
By definition, $\|I(t,\cdot)\|_{H_r(D)}^2=\sum_{k\geq 1}\lambda_k^r \big|I_k(t)\big|^2$, and hence
\begin{align}
\nonumber
\bE\big[\sup_{t\leq T}\|I(t,\cdot)\|_{H_r(D)}^2\big] & \leq \sum_{k\geq 1}\lambda_k^r \bE \big[\sup_{t\leq T}\big|I_k(t)\big|^2\big] \\
\label{bound-I}
& \leq 2T^2|D| D_b^2
\sum_{k\geq 1}\lambda_k^{r-\gamma} \big(1+\sup_{t \in [0,T]}\sup_{x\in \bR}\bE[|u(t,x)|^2]\big)<\infty.
\end{align}
In particular, $I(t,\cdot)\in H_r(D)$ for all $t \in [0,T]$, with probability 1.

\medskip

Next, we treat $Z$. 
We first prove that for any $t \geq 0$ fixed, 
\begin{equation}
\label{Z-in-H}
Z(t,\cdot) \in H_r(D) \quad \mbox{a.s.}
\end{equation} 
Then, by Corollary \ref{point-rep}, for any $t \geq 0$ fixed, we have the following pointwise representation:
\begin{equation}
\label{point-rep-Z}
Z(t,x)=\sum_{k\geq 1}Z_k(t) e_k(x) \quad \mbox{for all $x \in D$ \ a.s.} 
\end{equation}
We denote by $\Omega_t^*$ the event of probability 1 where this holds.

We prove \eqref{Z-in-H}. By the stochastic Fubini theorem (Theorem A.3 of \cite{CDH19}) and \eqref{Fourier-G},
the $k$-th Fourier coefficient of $Z(t,\cdot)$ is:
\begin{align*}
Z_k(t) & := \cF_k [Z(t,\cdot)] =\langle Z(t,\cdot),e_k\rangle_{L^2(D)} \\ 
& =\int_0^t \int_D \left(\int_D G_{t-s}(x,y)e_k(x) dx\right) \sigma\big(u(s,y)\big)L(ds,dy)\\
&
=\frac{1}{{\lambda_k^{\gamma/2}}} \int_0^t \int_{D} \sin\big(\lambda_k^{\gamma/2}(t-s)\big) e_k(y) \sigma\big(u(s,y)\big)L(ds,dy).
\end{align*}
Using the formula $\sin(a-b)=\sin a \cos b-\cos a \sin b$, we write
\begin{equation}
\label{def-FkZ}
Z_k(t)=\frac{1}{\lambda_k^{\gamma/2}} \big[\sin(\lambda_k^{\gamma/2}t) M^{(k)}(t) -
\cos(\lambda_k^{\gamma/2}t) N^{(k)}(t)\big],
\end{equation}
where 
\begin{align*}
M^{(k)}(t):=\int_0^t \int_{D} \cos\big(\lambda_k^{\gamma/2}s\big) e_k(y) \sigma\big(u(s,y)\big)L(ds,dy) \\
N^{(k)}(t):=\int_0^t \int_{D} \sin\big(\lambda_k^{\gamma/2}s \big) e_k(y) \sigma\big(u(s,y)\big)L(ds,dy).
\end{align*}
By the isometry property \eqref{isometry}, the linear growth property \eqref{lin-grow} of $\sigma$, and \eqref{def-KT}, we have
\begin{align*}
\bE|M^{(k)}(t)|^2 & =m_2 \int_0^t \int_{D} \cos^2\big(\lambda_k^{\gamma/2}s\big) e_k^2(y) \bE|\sigma\big(u(s,y)\big)|^2dyds\\
 & \leq 2m_2 D_{\sigma}^2 \int_0^t \int_{D} e_k^2(y)\big(1+ \bE|u(s,y)|^2\big) dyds \leq 2m_2 D_{\sigma}^2 T K_T.
\end{align*}
The same inequality holds for $N^{(k)}(t)$. Hence,
\[
\bE|Z_k(t)|^2 \leq \frac{2}{\lambda_k^{\gamma}} \Big(\bE |M^{(k)}(t)|^2 + \bE|N^{(k)}(t)|^2\Big)
\leq \frac{ 8m_2 D_{\sigma}^2 TK_T}{\lambda_k^{\gamma}},
\]
and $\bE\|Z(t,\cdot)\|_{H_r(D)}^2=\sum_{k\geq 1} \lambda_k^{r} \bE|Z_k(t)|^2 <\infty$, since $r<\gamma-\frac{d}{2}$. This proves \eqref{Z-in-H}.

\medskip
Next, we construct a modification $\widetilde{Z}$ of $Z$.
The martingales $M^{(k)}$ and $N^{(k)}$ have c\`adl\`ag modifications
$\widetilde{M}^{(k)}$ and $\widetilde{N}^{(k)}$; see e.g. Corollary 1 (Section 1.2) of \cite{protter}. We define:
\begin{equation}
\label{def-tZ}
\widetilde{Z}_k(t):=\frac{1}{\lambda_k^{\gamma/2}} \big[\sin(\lambda_k^{\gamma/2}t) \widetilde{M}^{(k)}(t)-
\cos(\lambda_k^{\gamma/2}t) \widetilde{N}^{(k)}(t)\big].
\end{equation}
Hence, the map $t \mapsto \widetilde{Z}_k(t)$ is c\`adl\`ag on $[0,T]$ (recall that a {\em c\`adl\`ag} function is a right continuous function with left limits). It follows that
\begin{align}
\label{Z-MN}
\bE\big[\sup_{t\leq T}|\widetilde{Z}_k(t)|^2\big] \leq \frac{2}{\lambda_k^{\gamma}} \left\{ 
\bE\big[\sup_{t\leq T}|\widetilde{M}^{(k)}(t)|^2\big]+\bE\big[\sup_{t\leq T}|\widetilde{N}^{(k)}(t)|^2\big]\right\}.
\end{align}
We apply Doob's maximal inequality (see, e.g., \cite[Thm. 3.8]{Karatzas}): for any non-negative and right-continuous submartingale $(M_t)_{t\in [0,T]}$ such that $\bE [M_T^p]<\infty$ for some $p>1$, it holds
\[
\bE\big[\sup_{t\leq T} M_t^p\big] \leq \left( \frac{p}{p-1}\right)^p\bE[M_T^p]. 
\]
Hence, using the linear growth property of $\sigma$, we obtain:
\begin{align}
\nonumber
\bE\big[\sup_{t\leq T} |\widetilde{M}^{(k)}(t)|^2 \big]  & \leq 4 \bE\big[|\widetilde{M}^{(k)}(T)|^2 \big]\\ \nonumber 
& =
4m_2 \bE \left[\int_0^T \int_{D}  \cos^2(\lambda_k^{\gamma/2} s) |e_k(y)|^2 \big|\sigma(u(s,y))\big|^2 dyds\right] \\
\label{bound-tM}
& \leq 8 m_2 D_{\sigma}^2  \bE\left[ \int_0^T \int_{D} |e_k(y)|^2 \big(1+|u(s,y)|^2 \big) dyds\right]\\
\nonumber
& \leq 8 m_2 D_{\sigma}^2  T \big(1+\sup_{t\in [0,T]}\sup_{x\in \bR}\bE[|u(t,x)|^2] \big).
\end{align}
The same inequality holds for $\widetilde{N}^{(k)}$.
Therefore, by \eqref{Z-MN}, we conclude that:
\begin{align*}
\bE\big[\sup_{t\leq T} |\widetilde{Z}_k(t)|^2 \big]   \leq \frac{32 m_2 D_{\sigma}^2  T}{\lambda_k^{\gamma}} \big(1+\sup_{t\in [0,T]}\sup_{x\in \bR}\bE[|u(t,x)|^2] \big)<\infty.
\end{align*}
It follows that on an event $\Omega^*$ of probability 1, $\sup_{t\leq T} |\widetilde{Z}_k(t)|^2<\infty$ and hence, $\widetilde{Z}_k(t) \in L^2(D)$ for all $t \in [0,T]$. On the event $\Omega^*$, we define
\[
\widetilde{Z}(t,x):=\sum_{k\geq 1} \widetilde{Z}_k(t) e_k(x) \quad \mbox{for any $t \in [0,T],x \in D$}.
\]
For any $t \in [0,T]$, consider the following event of probability 1:  
\[
\Omega_t=\bigcap_{k\geq 1}\big\{M^{(k)}(t)=\widetilde{M}^{(k)}(t) \ \mbox{and} \  N^{(k)}(t)=\widetilde{N}^{(k)}(t) \big\}.
\]
By \eqref{def-FkZ} and \eqref{def-tZ},
on the event $\Omega_t$, $\widetilde{Z}_k(t)=Z_k(t)$ for all $k\geq 1$. Using the pointwise representation \eqref{point-rep-Z} of $Z$, it follows that on the event $\Omega_t^{*} \cap \Omega_t \cap \Omega^*$, $Z(t,x)=\widetilde{Z}(t,x)$ for all $x \in D$. This means that for any $t \in [0,T]$,
\begin{equation}
\label{eq-Z-tZ}
\bP(Z(t,x)=\widetilde{Z}(t,x) \ \mbox{for all} \ x \in D)=1.
\end{equation}
By definition, $\|\widetilde{Z}(t,\cdot)\|_{H_r(D)}^2 =\sum_{k\geq 1}\lambda_k^r |\widetilde{Z}_k(t)|^2$, and hence
\begin{align}
\label{bound-tZ1}
\bE\big[\sup_{t\leq T} \|\widetilde{Z}(t,\cdot)\|_{H_r(D)}^2 \big] & \leq \sum_{k\geq 1}\lambda_k^{r}\bE \big[ \sup_{t\leq T} |\widetilde{Z}_k(t)|^2 \big] \\
\label{bound-tZ}
& \leq 32m_2 D_{\sigma}^2 T \sum_{k\geq 1}\lambda_k^{r-\gamma}\big(1+\sup_{t\in [0,T]}\sup_{x\in \bR}\bE[|u(t,x)|^2] \big)<\infty.
\end{align}
In particular, $\widetilde{Z}(t,\cdot)\in H_r(D)$ for all $t \in [0,T]$, with probability 1.
On the event $\Omega^*$, we define for any $t \in [0,T]$ and $x \in D$,
\begin{equation}
\label{def-tu}
\tu(t,x):=I(t,x)+\widetilde{Z}(t,x). 
\end{equation}
Relation
\eqref{u-tu} now follows from \eqref{eq-tx}, \eqref{def-tu} and \eqref{eq-Z-tZ}.
By \eqref{bound-I} and \eqref{bound-tZ},
\begin{align}
\label{u-IZ}
\bE\big[\sup_{t\leq T} \|\widetilde{u}(t,\cdot)\|_{H_r(D)}^2 \big] & \leq 2\left\{\bE\big[\sup_{t\leq T} \|I(t,\cdot)\|_{H_r(D)}^2 \big] +\bE\big[\sup_{t\leq T} \|\widetilde{Z}(t,\cdot)\|_{H_r(D)}^2 \big]\right\}\\
\nonumber
& \leq 2(2T^2 |D|D_b^2+32 m_2 D_{\sigma}^2 T) \sum_{k\geq 1}\lambda_k^{r-\gamma}(1+\sup_{t\in [0,T]}\sup_{x\in \bR}\bE[|u(t,x)|^2])\\
\nonumber
&=T\widehat{C}_T(1+K_T)=:\widetilde{C}_T.
\end{align}

\medskip

{\em Step 2.} For any $T>0$, let $\{\tu^{(T)}(t,x);t \in [0,T],x\in D\}$ be the process from {\em Step 1}. This means that $\bP\big(u(t,x)=\tu^{(T)}(t,x) \big)=1$ for any $t \in [0,T]$ and  $x \in D$, and
\[
\bE\big[ \sup_{t\leq T} \|\tu^{(T)}(t,\cdot)\|_{H_r(D)}^2\big] \leq \widetilde{C}_T.
\]

Define
\[
\tu(t,x):=\tu^{(N)}(t,x) \quad \mbox{if $t \in [N-1,N)$ and $x\in D$}.
\]
Then \eqref{u-tu} clearly holds. To prove that \eqref{sup-Hr} holds, let $T>0$ be arbitrary. Then there exists some integer $N\geq 1$ such that $T \in [N-1,N)$. It follows that
$\sup_{t\leq T}\|\tu(t,\cdot)\|_{H_r(D)}^2 =\max\{X_1,\ldots,X_N\}$,
where 
\begin{align*}
X_k &=\sup_{t\in [k-1,k)}\|\tu(t,\cdot)\|_{H_r(D)}^2 =\sup_{t\in [k-1,k)}\|\tu^{(k)}(t,\cdot)\|_{H_r}^2\quad \mbox{for $k\leq N-1$}, \\
X_N & =\sup_{t\in [N-1,T]}\|\tu(t,\cdot)\|_{H_r(D)}^2=\sup_{t\in [N-1,T]}\|\tu^{(N)}(t,\cdot)\|_{H_r(D)}^2.
\end{align*}
Note that $\bE[X_k]\leq \widetilde{C}_k \leq \widetilde{C}_T$ for any $k\leq N-1$, and $\bE[X_N]\leq \widetilde{C}_T$. Hence,
\[
\bE\big[\sup_{t\leq T}\|\tu(t,\cdot)\|_{H_r(D)}^2\big] =\bE\big[\max\{X_1,\ldots,X_N\}\big] \leq \sum_{k=1}^N \bE[X_k] \leq N \widetilde{C}_T \leq (T+1)\widetilde{C}_T.
\]

\medskip
{\em Step 3.}  In this step, we prove that $t \mapsto \|\tu(t,\cdot)\|_{H_r(D)}$ is c\`adl\`ag on $[0,T]$. Note that
\begin{align*}
\|\tu(t,\cdot)\|_{H_r(D)}^2 = \sum_{k\geq 1}\lambda_k^r |\cF_k[\tu(t,\cdot)]|^2 \quad  \mbox{and} \quad \cF_k[\tu(t,\cdot)]=I_k(t)+\widetilde{Z}_k(t).
\end{align*}
The map $t \mapsto I_k(t)$ is continuous, while $t \mapsto \widetilde{Z}_k(t)$ is  c\`adl\`ag.

Let $S_n(t)=\sum_{k=1}^n \lambda_k^r |\cF_k [\tu(t,\cdot)]|^2$.  Then $t \mapsto S_n(t)$ is c\`adl\`ag on $[0,T]$. The same argument as for \eqref{bound-I} and \eqref{bound-tZ} shows that
\begin{align*}
& \bE \big[\sup_{t\leq T}|S_n(t))-\|\tu(t,\cdot)\|_{H_r(D)}^2| \big] \\
& \quad \leq 2\big(2T^2 |D|D_b^2 +32 m_2 D_{\sigma}^2 T \big) \sum_{k\geq n+1}\lambda_k^{r-\gamma}\big(1+\sup_{t\in [0,T]}\sup_{x\in \bR}\bE[|u(t,x)|^2] \big),
\end{align*}
and the latter term tends to zero 
as $n\to \infty$. Hence, along a subsequence, $S_n(t) \to \|\tu(t,\cdot)\|_{H_r(D)}^2$ uniformly in $t \in [0,T]$, with probability 1. Since the uniform limit of c\`adl\`ag functions is c\`adl\`ag (see e.g. \cite[Thm. 43]{protter}), the map $t \mapsto \|\tu(t,\cdot)\|_{H_r(D)}^2 $ is c\`adl\`ag on $[0,T]$, with probability 1.

\end{proof}

Next, we will prove another result in which \eqref{sup-Hr} is replaced by a more powerful inequality, which does not depend on $K_T$. This new inequality (given by Theorem \ref{tu-th2} below) will be used in Section \ref{section-loc-Lip} in the case when $b$ and $\sigma$ are locally Lipschitz. To prove this inequality, we introduce two operators $\cT$ and $\widetilde{\cT}$. We explain this below.

\medskip

Let $\cL_{loc}^2$ be the set of all predictable processes $\phi=\{\phi(t,x);t\geq 0,x\in D\}$ such that
\[
\sup_{t \in [0,T]} \sup_{x\in D}\bE[|\phi(t,x)|^2]<\infty \quad \mbox{for all $T>0$}.
\]
We now define the operators $\cT$ and $\widetilde{\cT}$, and describe their relation with the solution $u$ and its modification $\tu$.

\medskip

a) {\em The operator $\cT$.} For any $\phi \in \cL_{loc}^2$, we let $\cT \phi=\cI \phi+\cZ \phi$, where
\begin{align*}
(\cI \phi)(t,x)&=\int_0^t \int_{D}G_{t-s}(x,y) b \big(\phi(s,y) \big)dyds\\
(\cZ \phi)(t,x)&=\int_0^t \int_{D}G_{t-s}(x,y) \sigma \big(\phi(s,y) \big)L(ds,dy).
\end{align*}
We consider the Fourier coefficients: for any $k\geq 1$,
\begin{align}
\label{def-Four-cI}
(\cI_k \phi)(t) &=\cF_k [(\cI \phi)(t,\cdot)]=\frac{1}{\lambda_k^{\gamma/2}} \int_0^t \int_{D}\sin \big(\lambda_k^{\gamma/2} (t-s)\big) e_k(y) b\big(\phi (s,y)\big) dyds\\
\nonumber
(\cZ _k \phi)(t) &= \cF_k [(\cZ \phi)(t,\cdot)]=\frac{1}{\lambda_k^{\gamma/2}} \big[\sin \big(\lambda_k^{\gamma/2} t \big) M_{\phi}^{(k)}(t) -\cos \big(\lambda_k^{\gamma/2} t\big) N_{\phi}^{(k)}(t) \big]
\end{align}
where 
\begin{align*}
M_{\phi}^{(k)}(t) &= \int_0^t \int_{D}\cos \big(\lambda_k^{\gamma/2} s\big) e_k(y) \sigma\big(\phi (s,y)\big) L(ds,dy)\\
N_{\phi}^{(k)}(t) & = \int_0^t \int_{D}\sin \big(\lambda_k^{\gamma/2} s \big)  e_k(y) \sigma\big(\phi (s,y)\big) L(ds,dy).
\end{align*}
Then $I=\cI u$, $Z=\cZ u$, $I_k=\cI_k u$, $Z_k=\cZ_ku$, $M^{(k)}=M^{(k)}_{u}$ and $N^{(k)}=N^{(k)}_{u}$,
 where $u$ is the solution of equation \eqref{fSWE} and $I,Z,I_k,Z_k,M^{(k)},N^{(k)}$ are defined in the proof of Theorem \ref{tu-th}.
Note that
\[
(\cT \phi)(t,x)=\sum_{k\geq 1}\big((\cI_k \phi)(t)+(\cZ_k \phi)(t)\big)e_k(x),
\] 
and hence, the Fourier coefficients of $(\cT \phi)(t,\cdot)$ are:
\[
\cF_k [(\cT \phi)(t,\cdot)]=(\cI_k \phi)(t)+(\cZ_k \phi)(t), \quad \mbox{for all} \ k\geq 1.
\]

b) {\em The operator $\widetilde{\cT}$.}
Define
\begin{equation}
\label{def-Four-cZ}
(\widetilde{\cZ _k} \phi)(t) =\frac{1}{\lambda_k^{\gamma/2}} \big[\sin \big(\lambda_k^{\gamma/2} t \big) \widetilde{M_{\phi}^{(k)}}(t) -\cos \big(\lambda_k^{\gamma/2} t\big) \widetilde{N_{\phi}^{(k)}}(t) \big],
\end{equation}
where $\widetilde{M_{\phi}^{(k)}}$ and  $\widetilde{N_{\phi}^{(k)}}$ are c\`adl\`ag modifications of the martingales $M_{\phi}^{(k)}$ and $N_{\phi}^{(k)}$. Let
\begin{equation}
\label{def-TT}
\widetilde{\cT} \phi=\cI \phi+\widetilde{\cZ} \phi
\quad \mbox{where} \quad
(\widetilde{\cZ} \phi)(t,x)=\sum_{k\geq 1} (\widetilde{\cZ _k} \phi) e_k(x).
\end{equation}
Then $\widetilde{Z}_k=\widetilde{\cZ}_k u$, where $\widetilde{Z}_k$ is given by \eqref{def-tZ}.
The Fourier coefficients of $(\widetilde{\cT} \phi)(t,\cdot)$ are:
\begin{equation}
\label{def-Fk-Tu}
\cF_k [(\widetilde{\cT} \phi)(t,\cdot)]=(\cI_k \phi)(t)+(\widetilde{\cZ_k} \phi)(t), \quad \mbox{for all} \ k\geq 1.
\end{equation}
Since $\bP \big((\cZ_k \phi)(t) = (\widetilde{\cZ_k} \phi)(t) \big)=1 $
for all $t\geq 0$ and $k\geq 1$, it follows that for any $t\geq 0$,
\begin{equation}
\label{T-tT}
\bP \big((\cT \phi)(t,x) = (\widetilde{\cT} \phi)(t,x) \quad \mbox{for all} \ x\in D \big)=1.
\end{equation}

c) {\em The solution $u$ and its modification $\tu$.}
Assume that $\gamma>d$, and let $u$ be a global solution of \eqref{fSWE}. This means that
\[
\bP\big(u(t,x)=(\cT u)(t,x) \big)=1 \quad \mbox{for all} \ t\geq 0,x \in D.
\]
By definition \eqref{def-tu} of $\tu$, we know that $\tu=\widetilde{\cT}u$. By \eqref{T-tT}, we have: for any $t\geq 0$,
\[
\bP \big((\cT u)(t,x) = \widetilde{u}(t,x) \quad \mbox{for all} \ x\in D \big)=1.
\]
The following result will be used in the proofs of Theorem \ref{tu-th2} and Lemma \ref{strong-uNN} below. It is a crucial statement, because it allows us to express the modification $\widetilde{u}$ as the sum of a Lebesgue integral and a stochastic integral, a.s, with the precise respective integrands.

\begin{lemma}
\label{lem-t-Tu}
Under the assumptions of Theorem \ref{tu-th},
\[
\bP\big(\tu(t,x)=(\widetilde{\cT} \widetilde{u})(t,x) \quad \mbox{for all $t\geq 0$ and $x\in D$} \big)=1.
\]
\end{lemma}

\begin{proof}
Using \eqref{def-Fk-Tu} for $\phi=u$ and $\phi=\tu$, we have:
\begin{align*}
&\bE \big[ \sup_{t\leq T} \| \tu(t,\cdot)-(\widetilde{\cT} \widetilde{u})(t,\cdot) \|_{H_r(D)}^2 \big]
\leq \sum_{k\geq 1}\lambda_k^r \bE \big[ \sup_{t\leq T} | \cF_k[(\widetilde{\cT} u)(t,\cdot)]-\cF_k[(\widetilde{\cT} \widetilde{u})(t,\cdot)] |^2 \big]\\
& \quad \leq 2 \left\{ \sum_{k\geq 1}\lambda_k^r \bE \big[ \sup_{t\leq T} | (\cI_k u)(t) -(\widetilde{\cI}_k u)(t) |^2 \big]+ \sum_{k\geq 1}\lambda_k^r \bE \big[ \sup_{t\leq T} | (\widetilde{\cZ}_k u)(t) -(\widetilde{\cZ}_k \tu)(t)
|^2 \big] \right\}\\
& \quad \leq 2\big(T|D| L_b^2+16 m_2 L_{\sigma}^2\big)
\sum_{k\geq 1}\lambda_k^{r-\gamma} \sup_{t \in [0,T]}\sup_{x\in \bR}\bE[|u(t,x)-\tu(t,x)|^2]=0,
\end{align*}
where the last inequality is proved similarly to \eqref{bound-I} and \eqref{bound-tZ}, using the Lipschitz property of $b$ and $\sigma$. Hence, by Theorem \ref{embed-th}, we obtain that, for any $T>0$,
\[
\bP\big(\tu(t,x)=(\widetilde{\cT} \widetilde{u})(t,x) \quad \mbox{for all $t\in[0,T]$ and $x\in D$} \big)=1.
\]
Finally, we take the intersection for all $T \in \bQ_{+}$ of these events of probability 1.
\end{proof}

\medskip

\begin{theorem}
\label{tu-th2}
Under the assumptions of Theorem \ref{tu-th}, for any $T>0$,
\[
\bE\big[\sup_{t \in [0,T]}\|\tu(t,\cdot)\|_{H_r(D)}^2 \big]\leq T \widehat{C}_{T} \exp\big(T \widehat{C}_T  \cC_{\infty}^2 \big),
\]
where $\widehat{C}_T$ is given by \eqref{def-hatC} and $\cC_{\infty}$ is the embedding constant from Theorem \ref{embed-th}.
\end{theorem}

\begin{proof} On the event of probability 1 given by Lemma \ref{lem-t-Tu}, $\widetilde{u}=\widetilde{\cT} \widetilde{u}=\cI \widetilde{u}+\widetilde{\cZ} \widetilde{u}$. Hence,
\begin{equation}
\label{sup-tu}
\bE\big[\sup_{t\leq T} \|\tu(t,\cdot)\|_{H_r(D)}^2 \big] \leq 2 \left\{ \bE\big[\sup_{t\leq T} \|(\cI \tu)(t,\cdot)\|_{H_r(D)}^2 \big] +  \bE\big[\sup_{t\leq T} \|(\widetilde{\cZ} \tu)(t,\cdot)\|_{H_r(D)}^2 \big]\right\}.
\end{equation}
We examine separately the two terms. 

We treat $\cI \tu$ first. Recall that $(\cI \tu)(t,\cdot)$ has Fourier coefficients $\{(\cI_k \tu)(t)\}_{k\geq 1}$ given by \eqref{def-Four-cI} with $\phi=\tu$.
Similarly to \eqref{bound-FI},
\begin{align*}
\sup_{t\leq T }\big|(\cI_k \tu)(t)\big|^2 & \leq    \frac{2T|D| D_b^2}{\lambda_k^{\gamma}} \int_0^T \int_{D} |e_k(y)|^2 \big(1+|\tu(s,y|^2\big) dyds\\
& \leq \frac{2T|D| D_b^2}{\lambda_k^{\gamma}} \int_0^T  \big(1+\sup_{y\in D}|\tu(s,y|^2\big) ds,
\end{align*}
where for the second line we used the fact that $\int_{D}|e_k(y)|^2 dy=1$. By Theorem \ref{embed-th},
\begin{equation}
\label{u-tu-ineq}
\sup_{y\in D}|\tu(s,y)|^2 \leq \cC_{\infty}^2 \|\tu(s,\cdot)\|_{H_r(D)}^2.
\end{equation}
Hence,
\begin{align*}
\bE \big[\sup_{t\leq T} \big|(\cI_k \tu)(t)\big|^2 \big]\leq \frac{2T|D|D_b^2}{\lambda_k^{\gamma}}\int_0^T \big(1+\cC_{\infty}^2 \bE[\|\tu(s,\cdot)\|_{H_r(D)}^2] \big)ds,
\end{align*}
and
\begin{align}
\label{bound-I2}
\bE\big[\sup_{t\leq T}\|(\cI \tu) (t,\cdot)\|_{H_r(D)}^2 \big]\leq 2T |D| D_b^2 \sum_{k\geq 1}\lambda_k^{r-\gamma}\int_0^T \big(1+\cC_{\infty}^2 \bE[\|\tu(s,\cdot)\|_{H_r(D)}^2] \big)ds.
\end{align}

Next, we examine $\widetilde{\cZ}\tu$.  Recall that $(\widetilde{\cZ} \tu)(t,\cdot)$ has Fourier coefficients $\{(\cZ_k \tu)(t)\}_{k\geq 1}$ given by \eqref{def-Four-cZ} with $\phi=\tu$.
Similarly to \eqref{bound-tM},
\begin{align*}
\bE\big[\sup_{t\leq T} |\widetilde{M}_{\tu}^{(k)}(t)|^2 \big]  
& \leq 8 m_2 D_{\sigma}^2  \bE \left[\int_0^T \int_{D} |e_k(y)|^2 \big(1+|\tu(s,y)|^2 \big) dyds\right]\\
&\leq 8 m_2 D_{\sigma}^2  \bE\left[ \int_0^T   \big(1+\sup_{y\in D}|\tu(s,y)|^2 \big) ds\right],
\end{align*}
where for the second line we used the fact that $\int_{D}|e_k(y)|^2dy=1$.
Using again the embedding inequality \eqref{u-tu-ineq}, it follows that
\begin{align*}
\bE\big[\sup_{t\leq T} |\widetilde{M}_{\tu}^{(k)}(t)|^2 \big]  
& \leq 8 m_2 D_{\sigma}^2 \int_0^T \big(1+\cC_{\infty}^2\bE[\|\tu(s,\cdot)\|_{H_r(D)}^2] \big) ds.
\end{align*}
The same inequality holds also for $\widetilde{N}_{\tu}^{(k)}$.
Therefore, similarly to \eqref{Z-MN}, we conclude that:
\begin{align*}
\bE\big[\sup_{t\leq T} |(\widetilde{\cZ}_k\tu)(t)|^2 \big]  \leq \frac{32 m_2 D_{\sigma}^2}{\lambda_k^{\gamma}} \int_0^T \big(1+\cC_{\infty}^2\bE[\|\tu(s,\cdot)\|_{H_r(D)}^2]\big) ds.
\end{align*}
Similarly to \eqref{bound-tZ1}, we infer that
\begin{align}
\label{bound-Z2}
\bE\big[\sup_{t\leq T} \|(\widetilde{\cZ} \tu)(t,\cdot)\|_{H_r(D)}^2 \big] 
& \leq 32m_2 D_{\sigma}^2  \sum_{k\geq 1}\lambda_k^{r-\gamma}  \int_0^T \big(1+\cC_{\infty}^2\bE[\|\tu(s,\cdot)\|_{H_r(D)}^2] \big) ds.
\end{align}
Using \eqref{sup-tu}, \eqref{bound-I2} and \eqref{bound-Z2}, it follows that:
\begin{align*}
\bE\big[\sup_{t\leq T} \|\widetilde{u}(t,\cdot)\|_{H_r(D)}^2 \big] 
\leq  \widehat{C}_T\int_0^T \big(1+\cC_{\infty}^2\bE[\|\tu(s,\cdot)\|_{H_r(D)}^2 ]\big) ds,
\end{align*}
where $\widehat{C}_T$ is given by \eqref{def-hatC}. It follows that for any $T>0$,
\[
\bE\big[\sup_{t\leq T} \|\tu(t,\cdot)\|_{H_r(D)}^2 \big] 
\leq  \widehat{C}_T T+\widehat{C}_T \cC_{\infty}^2 \int_0^T \bE\big[\sup_{\rho \leq s}\|\tu(\rho,\cdot)\|_{H_r(D)}^2\big]  ds.
\]
The conclusion follows by Gronwall's lemma.
\end{proof}

\subsection{Proof of Theorem \ref{main}}
\label{section-loc-Lip}

In this subsection, we give the proof of Theorem \ref{main}. We fix $r$ such that $\frac{d}{2}<r<\gamma-\frac{d}{2}$, and we let $\cC_{\infty}=\cC_{\infty}(r)$ be the embedding constant  of $H_r(D)$ into $L^{\infty}(D)$; see Theorem \ref{embed-th}.

For any $N\geq 1$, we consider the truncated functions:
\[
b_{N}(\xi)=\left\{
\begin{array}{ll} 
b(\xi) & \mbox{if $|\xi| \leq \cC_{\infty} N$} \\
b(\cC_{\infty}N) & \mbox{if $\xi>\cC_{\infty} N$} \\
b(-\cC_{\infty}N) & \mbox{if $\xi<-\cC_{\infty} N$}
\end{array} \right.
\quad \mbox{and} \quad 
\sigma_{N}(\xi)=\left\{
\begin{array}{ll} 
\sigma(\xi) & \mbox{if $|\xi| \leq \cC_{\infty} N$} \\
\sigma(\cC_{\infty}N) & \mbox{if $\xi>\cC_{\infty} N$} \\
\sigma(-\cC_{\infty}N) & \mbox{if $\xi<-\cC_{\infty} N$}
\end{array} \right.
\]
Note that $b_N$ and $\sigma_N$ are globally Lipschitz, with Lipschitz constants:
\[
L_{N}^{(b)}:=L_{b,\cC_{\infty}N} \quad \mbox{and} \quad L_{N}^{(\sigma)}:=L_{\sigma,\cC_{\infty}N},
\]
the respective Lipschitz constants of $b$ and $\sigma$ on the interval $[-\cC_{\infty} N,\cC_{\infty} N]$. The functions $b_N$ and $\sigma_{N}$ also have linear growth, with the {\em same} constants $D_b$ and $D_{\sigma}$ as $b$ and $\sigma$: for any $\xi \in \bR$,
\[
|b_N(\xi)|\leq D_b(1+|\xi|) \quad \mbox{and} \quad |\sigma_N(\xi)|\leq D_{\sigma}(1+|\xi|).
\]

We consider equation \eqref{fSWE} with $(b,\sigma)$ replaced by $(b_N,\sigma_N)$:  
\begin{equation}
\label{fSWE-N}
	\begin{cases}
		\dfrac{\partial^2 u}{\partial t^2} (t,x)
		=  -(-\Delta)^{\gamma}u(t,x) + b_N\big(u(t,x)\big) + \sigma_N\big(u(t,x)\big) \dot{L}(t,x), \
		t>0, \ x \in D, \\
		u(0,x) = 0, \quad \dfrac{\partial u}{\partial t}(0,x)=0, \quad x \in D \\
u(t,x)=0, \quad  t > 0, \ x \in \partial D
	\end{cases}
\end{equation}
We denote by $u_N$ the global solution of equation \eqref{fSWE-N}:
for any $t\geq 0$ and $x \in D$, 
\begin{equation}
\label{def-uN}
u_N(t,x)=\int_0^t \int_{D}G_{t-s}(x,y)b_N\big(u_N(s,y)\big)dyds+ \int_0^t \int_{D}
G_{t-s}(x,y)\sigma_N\big(u_N(s,y)\big)L(ds,dy) \quad \mbox{a.s.}
\end{equation}

Let $\tu_N$ be the process given by Theorem \ref{tu-th} (a modification of $u_N$). Then, by Theorems \ref{tu-th} and \ref{tu-th2},
\begin{description}
\item[a)] $\mathbb{P}\big(u_N(t,x)=\tu_N(t,x) \ \mbox{for all} \ x \in D\big)=1$ for any $t\geq 0$; in particular, $\tu_N$ is a modification of $u_N$ and satisfies \eqref{def-uN};
\item[b)] 
$\bE \big[ \sup_{t\leq T} \|\tu_N(t,\cdot) \|_{H_r(D)}^2 \big]
 \leq (T+1)T \widehat{C}_T (1+K_{T,N})$ where $\widehat{C}_T$ is given by \eqref{def-hatC} and
    \[
    K_{T,N}:=\sup_{t\in [0,T]} \sup_{x \in D}\bE[|u_{N}(t,x)|^2];
    \]
In particular, $\bP\big(\tu_N(t,\cdot) \in H_r(D) \ \mbox{for all $t\geq 0$}\big)=1$.
 
\item[c)] the map $t \mapsto \|\tu_N(t,\cdot)\|_{H_r(D)}$ is c\`adl\`ag on $\bR_{+}$ with probability 1;
\item[d)] $\bE\big[ \sup_{t\leq T} \|\tu_N(t,\cdot) \|_{H_r(D)}^2 \big]\leq T \widehat{C}_T \exp(T \widehat{C}_T \cC_{\infty}^2)$ for any $T>0$, where 
    $\widehat{C}_T$ is given by \eqref{def-hatC}.
\end{description}

Define
\[
\tau_N=\inf\left\{t>0; \|\tu_N(t,\cdot)\|_{H_r(D)} > N \right\} \quad (\inf \emptyset=\infty).
\]
By Lemma \ref{stop}, $\tau_N$ is a stopping time with respect to the filtration $(\cF_t)_{t\geq 0}$.
Note that
\[
 \|\tu_N(t,\cdot)\|_{H_r(D)} \leq N \quad \mbox{if} \quad t< \tau_N.
\]
Since $r>d/2$, by Theorem \ref{embed-th}, it follows that if $t<\tau_N$, then for any $x \in D$,
\begin{equation}
\label{bound-tu}
|\tu_N(t,x)| \leq \|\tu_N(t,\cdot)\|_{L^{\infty}(D)} \leq \cC_{\infty}\|\tu_N(t,\cdot)\|_{H_r(D)} \leq \cC_{\infty}N,
\end{equation}
Then, using the definitions of $b_N$ and $\sigma_N$, we infer that:
\begin{equation}
\label{bN-b}
b_{N}\big(\tu_{N}(t,x)\big)=b\big(\tu_{N}(t,x)\big) \ \mbox{for all $x \in D$}, \quad \mbox{if} \ t<\tau_N,
\end{equation}
\begin{equation} 
\label{sN-s}
\sigma_{N}\big(\tu_{N}(t,x)\big)=\sigma\big(\tu_{N}(t,x)\big) \ \mbox{for all $x \in D$}, \quad \mbox{if} \ t<\tau_N.
\end{equation}

\medskip

In addition to the local property for the stochastic integral with respect to $L$, given by Lemma \ref{local-cor}, we will use the following (obvious) local property of the Lebesgue integral: 
\begin{equation}
\label{local-Leb}
1_{\{t<\tau\}} \int_0^t \int_{D} X(s,y)dyds=1_{\{t<\tau \}} \int_0^t \int_{D} X(s,y)1_{\{s<\tau \}}dyds,
\end{equation}
for any function $\tau:\Omega \to \bR_{+}$ such that $s \mapsto 1_{\{s<\tau(\omega)\}}$ is measurable on $\bR_{+}$, for any $\omega \in \Omega$.

\medskip

The following result deals with the consistency in the definition of $\tu_{N}$. Though later on we will indeed prove a stronger statement, we include its proof because some parts of it will be used in other proofs in the sequel.

\begin{lemma}
\label{lem-uN-u}
If $\gamma > d$ and Assumption \ref{assumpt} holds, then for any $N \geq 1$, $t>0$ and $x \in D$,
\begin{equation}
\label{uN-u}
\tu_N(t,x)=\tu_{N+1}(t,x) \quad \mbox{a.s. on $\{t<\tau_N\}$}.
\end{equation}
\end{lemma}

\begin{proof} In view of Lemma \ref{fact1-lem}, we have to prove that: for any $N \geq 1$, $t>0$ and $x \in D$, 
\begin{equation}
\label{uN-u1}
1_{[0,\tau_N)}(t) \big(\tu_N(t,x)-\tu_{N+1}(t,x)\big)  =0 \quad \mbox{a.s.}
\end{equation}
We fix $(t,x)$. We write \eqref{def-uN} for $\tu_N(t,x)$ and for $\tu_{N+1}(t,x)$, we take the difference of these two equations, and then we multiply by $1_{[0,\tau_N)}(t)$. We obtain:
\begin{align*}
& 1_{[0,\tau_N)}(t) \big(\tu_N(t,x)-\tu_{N+1}(t,x)\big)\\
& \quad =1_{[0,\tau_N)}(t) \int_0^t \int_{D} G_{t-s}(x,y) \big(b_N(\tu_N(s,y))-b_{N+1}(\tu_{N+1}(s,y))\big)dyds\\
& \quad \quad + 1_{[0,\tau_N)}(t) \int_0^t \int_{D} G_{t-s}(x,y) \big(\sigma_N(\tu_N(s,y))-\sigma_{N+1}(\tu_{N+1}(s,y))\big)L(ds,dy)\\
& \quad =1_{[0,\tau_N)}(t) \int_0^t \int_{D} G_{t-s}(x,y) \big(b_N(\tu_N(s,y))-b_{N+1}(\tu_{N+1}(s,y))\big)1_{\{s<\tau_N \}}dyds\\
& \quad \quad + 1_{[0,\tau_N)}(t) \int_0^t \int_{D} G_{t-s}(x,y) \big(\sigma_N(\tu_N(s,y))-\sigma_{N+1}(\tu_{N+1}(s,y))\big)1_{\{s<\tau_N \}}L(ds,dy),
\end{align*}
where in the last line we used the local property given by \eqref{local-Leb} and Lemma
\ref{local-cor}.
If $t<\tau_N$, then  $|\tu_N(t,x)| \leq \cC_{\infty}N$ for any $x \in D$ (by \eqref{bound-tu}),
and hence,
\begin{align}
\label{bNN}
b_N\big(\tu_N(t,x)\big)& =b_{N+1}\big(\tu_N(t,x)\big)=b\big(\tu_N(t,x)\big) \quad \mbox{for all $x \in D$}, \\
\label{sNN}
\sigma_N\big(\tu_N(t,x)\big) & =\sigma_{N+1}\big(\tu_N(t,x)\big)=\sigma\big(\tu_N(t,x)\big) \quad \mbox{for all $x \in D$}.
\end{align}
It follows that 
\begin{align}
\nonumber
& 1_{[0,\tau_N)}(t) \big(\tu_N(t,x)-\tu_{N+1}(t,x)\big)\\
\nonumber
& \quad =1_{[0,\tau_N)}(t) \int_0^t \int_{D} G_{t-s}(x,y) \big(b_{N}\big(\tu_{N}(s,y)\big)-b_{N}\big(\tu_{N+1}(s,y)\big)\big)1_{\{s<\tau_N \}}dyds\\
\nonumber
& \quad \quad +1_{[0,\tau_N)}(t) \int_0^t \int_{D} G_{t-s}(x,y) \big(\sigma_{N}\big(\tu_{N}(s,y)\big)-\sigma_{N}\big(\tu_{N+1}(s,y)\big)\big)1_{\{s<\tau_N \}} L(ds,dy)\\
\label{uN-incr}
& \quad =: T_1+T_2.
\end{align}
We treat $T_1$ first. We take the square, we bound $1_{[0,\tau_N)}(t)$ by 1, and we use the Cauchy-Schwarz inequality for the $dsdy$ integral. Then we take expectation. We get:
\begin{align*}
\bE[|T_1|^2] & \leq  t |D|\int_0^t \int_{D} G_{t-s}^2(x,y) \bE\big[\big(b_{N}\big(\tu_{N}(s,y)\big)-b_{N}\big(\tu_{N+1}(s,y)\big)\big)^2 1_{\{s<\tau_N \}}\big] dyds \\
& \leq  t |D| \ (L_{N}^{(b)})^2 \int_0^t \int_{D} G_{t-s}^2(x,y) \bE\big[\big(\tu_{N}(s,y)-\tu_{N+1}(s,y)\big)^2 1_{\{s<\tau_N \}}\big] dyds\\
& \leq  t |D| \ (L_{N}^{(b)})^2\int_0^t \sup_{y\in \bR}\bE\big[\big(\tu_{N}(s,y)-\tu_{N+1}(s,y)\big)^2 1_{\{s<\tau_N \}}\big] \left( \sup_{x\in D}\int_{D} G_{t-s}^2(x,y) dy \right) ds,
\end{align*}
where we have used that $b_{N+1}$ is globally Lipschitz.
Similarly,
\begin{align*}
\bE[|T_2|^2] & \leq  m_2 \int_0^t \int_{D} G_{t-s}^2(x,y) \bE\big[\big(\sigma_{N}\big(\tu_{N}(s,y)\big)-\sigma_{N}\big(\tu_{N+1}(s,y)\big)\big)^2 1_{\{s<\tau_N \}}\big] dyds \\
& \leq   m_2 (L_{N}^{(\sigma)})^2\int_0^t \int_{D} G_{t-s}^2(x,y) \bE\big[\big(\tu_{N}(s,y)-\tu_{N+1}(s,y)\big)^2 1_{\{s<\tau_N \}}\big] dyds\\
& \leq  m_2 (L_{N}^{(\sigma)})^2\int_0^t \sup_{y\in \bR}\bE\big[\big(\tu_{N}(s,y)-\tu_{N+1}(s,y)\big)^2 1_{\{s<\tau_N \}}\big]  \left(\sup_{x\in D}\int_{D} G_{t-s}^2(x,y) dy\right)ds.
\end{align*}
We denote
\[
H(t):=\sup_{x\in D}\bE\big[\big(u_N(t,x)-u_{N+1}(t,x)\big)^2 1_{ \{t<\tau_N)\} }  \big].
\]
Hence, we have proved that, for any $t \in [0,T]$,
\[
H(t) \leq C_{T,N}\int_0^t H(s) J(t-s)ds,
\]
where 
\begin{equation}
\label{def-CJ}
C_{T,N}=2T |D| (L_{N}^{(\sigma)})^2+m_2 (L_{N}^{(\sigma)})^2 \quad \mbox{and} \quad J(t):=\sup_{x\in D}\int_{D} G_{t}^2(x,y) dy.
\end{equation}
We observe that by \eqref{square-G}, $\int_0^T J(t)dt<\infty$. 
By Lemma 15 of \cite{dalang99}, $H(t)=0$ for any $t \in [0,T]$. Hence, for any $t \in [0,T]$ and $x \in D$,
\[
\bE\big[\big(u_N(t,x)-u_{N+1}(t,x)\big)^2 1_{ \{t<\tau_N)\} } \big]=0.
\]
Thus, we conclude that \eqref{uN-u1} holds.
\end{proof}

For our developments, we will need a stronger property, which is stated by the following lemma. Note that a similar property is listed on p. 549 of \cite{DKZ19} in the case of the stochastic heat equation with space-time Gaussian white noise.

\begin{lemma}
\label{strong-uNN}
If $\gamma > d$ and Assumption \ref{assumpt} holds, then for all $N \geq 1$,
\begin{equation}
\label{strong-eq}
\bP\big(\tu_N(t,x)=\tu_{N+1}(t,x) \ \mbox{for all $t\in [0,\tau_N)$ and $x \in D$}\big)=1.
\end{equation}
\end{lemma}

\begin{proof}
Let $\widetilde{\cT}_N$ be the operator $\widetilde{\cT}$ defined by \eqref{def-TT} with $(b,\sigma)$ replaced by $(b_N,\sigma_N)$. By Lemma \ref{lem-t-Tu},
\[
\bP\big( \tu_N(t,x)=(\widetilde{\cT}_N \tu_N)(t,x) \quad \mbox{for all $t \geq 0$ and $x \in D$} \big)=1.
\]
Let $B_N(t,x)=(\widetilde{\cT}_N \tu_N)(t,x) -(\widetilde{\cT}_{N+1} \tu_{N+1})(t,x)$. Then
\[
\bP\big( \tu_N(t,x)-\tu_{N+1}(t,x)=B_N (t,x) \quad \mbox{for all $t \geq 0$ and $x \in D$} \big)=1,
\]
and hence
\begin{equation}
\label{uN-BN}
\bE\big[\sup_{t\leq T}\|\tu_N(t,\cdot)-\tu_{N+1}(t,\cdot) \|_{H_r(D)}^2 1_{\{t<\tau_N\}}\big]=
\bE\big[\sup_{t\leq T}\|B_N(t,\cdot) \|_{H_r(D)}^2 1_{\{t<\tau_N\}}\big].
\end{equation}
Note that $\|B_N(t,\cdot) \|_{H_r(D)}^2=\sum_{k\geq 1}\lambda_k^r \big(H_k^N(t)\big)^2$, where
\[
H_k^N(t)=\cF_k[(\widetilde{\cT}_N \tu_N)(t,\cdot)]-\cF_k[(\widetilde{\cT}_{N+1} \tu_{N+1})(t,\cdot)].
\]
Hence,
\begin{equation}
\label{bound-BN}
\bE \big[\sup_{t\leq T} \|B_N(t,\cdot) \|_{H_r(D)}^2 1_{\{t<\tau_N\}} \big] \leq \sum_{k\geq 1}\lambda_k^r \bE \big[\sup_{t\leq T} \big(H_k^N(t)\big)^2 1_{\{t<\tau_N\}} \big].
\end{equation}
We have:
\begin{align*}
\cF_k[(\widetilde{\cT}_{N} \tu_{N})(t,\cdot)] &=\frac{1}{\lambda_k^{\gamma/2}} \int_0^t \int_{D}
\sin\big(\lambda_k^{\gamma/2}(t-s)\big)e_k(y) b_{N}\big(\tu_{N}(s,y)\big)dyds\\
& \quad  + \frac{1}{\lambda_k^{\gamma/2}} \sin(\lambda_k^{\gamma/2}t) \int_0^t \int_{D}
\cos(\lambda_k^{\gamma/2}s) e_k(y) \sigma_{N}\big(\tu_{N}(s,y)\big)L(ds,dy)\\
&  \quad - \frac{1}{\lambda_k^{\gamma/2}} \cos(\lambda_k^{\gamma/2}t ) \int_0^t \int_{D}
\sin(\lambda_k^{\gamma/2}s )e_k(y) \sigma_{N}\big(\tu_{N}(s,y)\big)L(ds,dy),
\end{align*}
with the convention that the stochastic integrals above have c\`adl\`ag sample paths.

Recall that if $t<\tau_N$, then \eqref{bNN} and \eqref{sNN} hold, and we can replace $(b_{N},\sigma_{N})$ by $(b_{N+1},\sigma_{N+1})$. More precisely, using the same argument
as in the proof of Lemma \ref{lem-uN-u} (based on the local properties given by \eqref{local-Leb} and Lemma \eqref{local-cor}), we infer that on the event $\{t<\tau_N\}$,
\begin{align*}
\cF_k[(\widetilde{\cT}_{N+1} \tu_{N+1})(t,\cdot)] &=\frac{1}{\lambda_k^{\gamma/2}} \int_0^t \int_{D}
\sin\big(\lambda_k^{\gamma/2}(t-s)\big)e_k(y) b_{N}\big(\tu_N(s,y)\big)dyds\\
&  + \frac{1}{\lambda_k^{\gamma/2}} \sin(\lambda_k^{\gamma/2}t) \int_0^t \int_{D}
\cos(\lambda_k^{\gamma/2}s) e_k(y) \sigma_{N}\big(\tu_{N}(s,y)\big)L(ds,dy)\\
&  - \frac{1}{\lambda_k^{\gamma/2}} \cos(\lambda_k^{\gamma/2}t ) \int_0^t \int_{D}
\sin(\lambda_k^{\gamma/2}s )e_k(y) \sigma_{N}\big(\tu_{N}(s,y)\big)L(ds,dy).
\end{align*}
It follows that
\begin{align*}
& 1_{\{t<\tau_N\}} \lambda_k^{\gamma/2} H_k^N(t)   \\
&   = 1_{\{t<\tau_N\}} \int_0^t \int_{D}
\sin\big(\lambda_k^{\gamma/2}(t-s)\big)e_k(y) \big[b_{N}\big(\tu_N(s,y)\big) - b_{N}\big(\tu_{N+1}(s,y)\big) \big] 1_{\{s<\tau_N\}}  dyds\\
&   + 1_{\{t<\tau_N\}}  \sin(\lambda_k^{\gamma/2}t)  \int_0^t \int_{D}
\cos(\lambda_k^{\gamma/2}s) e_k(y) \big[ \sigma_{N}\big(\tu_{N}(s,y)\big) -
 \sigma_{N}\big(\tu_{N+1}(s,y)\big)  \big]  1_{\{s<\tau_N\}}  L(ds,dy)\\
&   - 1_{\{t<\tau_N\}}  \cos(\lambda_k^{\gamma/2}t ) \int_0^t \int_{D}
\sin(\lambda_k^{\gamma/2}s )e_k(y) \big[ \sigma_{N}\big(\tu_{N}(s,y)\big)-\sigma_{N}\big(\tu_{N+1}(s,y)\big)\big]  1_{\{s<\tau_N\}}  L(ds,dy).
\end{align*}
Using the same argument as for \eqref{bound-I2} and \eqref{bound-tM}, and the Lipschitz property of $b_{N}$ and $\sigma_{N}$, we obtain that:
\begin{align*}
& \bE\big[\sup_{t\leq T} \big(H_k^N(t) \big)^2 1_{\{t<\tau_N\}} \big] \\
& \quad \leq \frac{2}{\lambda_k^{\gamma}} \Big(T|D| L_{N}^{(b)})^2 + 16 m_2 (L_{N}^{(\sigma)})^2 \Big) \int_0^t \int_{D}|e_k(y)|^2 \bE\big[(\tu_N(s,y)-\tu_{N+1}(s,y))^2 1_{\{s<\tau_N\}}\big]dyds.
\end{align*}
By Theorem \ref{embed-th},
\begin{align*}
(\tu_N(s,y)-\tu_{N+1}(s,y))^2 \leq \|\tu_N(s,\cdot)-\tu_{N+1}(s,\cdot)\|_{L^{\infty}(D)}^2 \leq  
\cC_{\infty}^2  \|\tu_N(s,\cdot)-\tu_{N+1}(s,\cdot)\|_{H_r(D)}^2.
\end{align*}
Using also the fact that $\int_{D}|e_k(y)|^2 dy=1$, it follows that
\begin{align*}
& \bE\big[\sup_{t\leq T} \big(H_k^N(t) \big)^2 1_{\{t<\tau_N\}} \big]\leq \\
& \quad \frac{2}{\lambda_k^{\gamma}} \Big(T|D| L_{N}^{(b)})^2 + 16m_2 (L_{N}^{(\sigma)})^2 \Big) \cC_{\infty}^2 \int_0^t  \bE\big[\sup_{\ell \leq s}\|\tu_N(\ell,\cdot)-\tu_{N+1}(\ell,\cdot)\|_{H_r(D)}^2 1_{\{s<\tau_N\}}\big]ds.
\end{align*}
Returning to \eqref{bound-BN}, and using also \eqref{uN-BN}, we obtain:
\begin{align*}
& \bE\big[\sup_{t\leq T} \| \tu_N(t,\cdot) -\tu_{N+1}(t,\cdot) \|_{H_r(D)}^2  1_{\{t<\tau_N\}} \big]\\
& \quad \leq C_N \int_0^t  \bE\big[\sup_{\ell \leq s}\|\tu_N(\ell,\cdot)-\tu_{N+1}(\ell,\cdot)\|_{H_r(D)}^2 1_{\{s<\tau_N\}}\big]ds,
\end{align*}
where 
\begin{equation}
\label{def-CN}
C_N=2\cC_{\infty}^2 \Big(T|D| L_{N}^{(b)})^2 + 16m_2 (L_{N}^{(\sigma)})^2 \Big) \sum_{k\geq 1}\lambda_k^{r-\gamma}.
\end{equation}
 By Gronwall lemma, 
\[
\bE\big[\sup_{t\leq T} \| \tu_N(t,\cdot) -\tu_{N+1}(t,\cdot) \|_{H_r(D)}^2  1_{\{t<\tau_N\}} \big]=0 \quad \mbox{for all $T>0$}.
\]
Hence, $\bP(\Omega_T)=1$ for any $T>0$, where
\[
\Omega_T:=\big\{\sup_{t\leq T} \| \tu_N(t,\cdot) -\tu_{N+1}(t,\cdot) 1_{\{t<\tau_N\}} \|_{H_r(D)}^2  =0 \big\}.
\]
By Theorem \ref{embed-th}, $\Omega_T \subseteq \Omega_T^*$, where
\begin{align*}
\Omega_T^*  & := \big\{\sup_{t\leq T} \| \tu_N(t,\cdot) -\tu_{N+1}(t,\cdot) 1_{\{t<\tau_N\}} \|_{L^{\infty}(D)}^2  =0 \big\}\\
&=\big\{\big(\tu_N(t,x) -\tu_{N+1}(t,x)\big) 1_{\{t<\tau_N\}}  =0 \quad \mbox{for all} \ t\leq T \ \mbox{and} \ x\in D\big\}.
\end{align*}
Finally,
\begin{align*}
\bigcap_{T\in \bQ_{+}}\Omega_T^* & =\big\{\big(\tu_N(t,x) -\tu_{N+1}(t,x)\big) 1_{\{t<\tau_N\}}  =0 \quad \mbox{for all} \ t\geq 0 \ \mbox{and} \ x\in D\big\}
\\
&=\big\{\tu_N(t,x) =\tu_{N+1}(t,x)   \quad \mbox{for all} \ t\in [0,\tau_N) \ \mbox{and} \ x\in D\big\},
\end{align*}
where for the last line we used the fact that $\bigcap_{t\geq 0}\{t\geq \tau_N\}=\emptyset$. 

\end{proof}

\medskip

\begin{lemma}
\label{tau-incr}
If $\gamma > d$ and Assumption \ref{assumpt} holds, then
for any $N \geq 1$, $\tau_N \leq \tau_{N+1}$ a.s.
\end{lemma}

\begin{proof}
Let $N \geq 1$ be arbitrary and $\Omega_N=\{\tau_N \leq \tau_{N+1}\}$. We will prove that
\[
\Omega_N^* \subseteq \Omega_N,
\]
where $\Omega_N^*$ is the event of probability 1 from Lemma \ref{strong-uNN}.

Let $\omega \in \Omega_N^*$ be arbitrary. Then
\begin{equation}
\label{tau-omega}
\tu_N(\omega,t,x)=\tu_{N+1}(\omega,t,x) \quad \mbox{for all $t \in [0,\tau_N(\omega))$ and $x \in D$}.
\end{equation}
Suppose by contradiction that $\omega \in \Omega_N^c$, i.e. $\tau_{N+1}(\omega)<\tau_N(\omega)$.
From \eqref{tau-omega}, we get:
\[
\tu_N(\omega,\tau_{N+1}(\omega),x)=\tu_{N+1}(\omega,\tau_{N+1}(\omega),x) \quad \mbox{for all $x \in D$}.
\]
Therefore,
\begin{equation}
\label{eq1}
\|\tu_N(\omega,\tau_{N+1}(\omega),\cdot)\|_{H_r(D)}=\|\tu_{N+1}(\omega,\tau_{N+1}(\omega),\cdot)\|_{H_r(D)}. 
\end{equation}
We apply Lemma \ref{cont-g-lem}.a) to the right-continuous function $t \mapsto \|\tu_{N+1}(\omega,t,\cdot)\|_{H_r(D)}$ and
\[
\tau_{N+1}(\omega)=\inf\left\{t>0; \|\tu_{N+1}(\omega,t,\cdot)\|_{H_r(D)} > N+1 \right\}.
\]
It follows that 
\begin{equation}
\label{eq2}
\|\tu_{N+1}(\omega,\tau_{N+1}(\omega,\cdot)\|_{H_r(D)} \geq N+1
\end{equation}
(this inequality is in fact valid for any $\omega \in \Omega$).
On the other hand, by the definition of $\tau_N$, 
\[
\|\tu_{N}(\omega,t,x) \|_{H_r(D)} \leq N \quad \mbox{for any $t<\tau_N(\omega)$}.
\]
Taking $t=\tau_{N+1}(\omega)$, we get:
\begin{equation}
\label{eq3}
\|\tu_{N}(\omega,\tau_{N+1}(\omega),x) \|_{H_r(D)} \leq N.
\end{equation}
Summarizing \eqref{eq1}, \eqref{eq2} \eqref{eq3}, we obtain:
\[
N+1\leq \|\tu_{N+1}(\omega,\tau_{N+1}(\omega),\cdot)\|_{H_r(D)} =\|\tu_N(\omega,\tau_{N+1}(\omega),\cdot)\|_{H_r(D)} \leq N,
\]
which is a contradiction. Hence, $\omega \in \Omega_N$.
\end{proof}

\medskip

\begin{lemma}\label{lem:1}
If $\gamma > d$ and Assumption \ref{assumpt} holds, then $\lim_{N \to \infty}\tau_{N}=\infty$ a.s. 
\end{lemma}

\begin{proof}
We apply Lemma \ref{cont-g-lem}.b) to the right-continuous function $t \mapsto \|\tu_N(t,\cdot)\|_{H_r(D)}$. Recalling the definition of $\tau_N$, we have:
\[
\{\tau_N \leq T\} \subset \left\{\sup_{t\leq T}\|\tu_N(t,\cdot)\|_{H_r(D)} \geq N \right\}.
\]
Using Chebyshev's inequality and property d) of $\tu_N$ mentioned at the beginning of this section, we get:
\begin{align*}
\bP(\tau_N \leq T) & \leq \bP\left(\sup_{t\leq T}\|\tu_N(t,\cdot)\|_{H_r(D)} \geq N \right)\\
& 
\leq \frac{1}{N^2} \bE\Big[\sup_{t\leq T}\|\tu_N(t,\cdot)\|_{H_r(D)}^2 \Big]\\
& \leq \frac{1}{N^2} T \widehat{C}_T \exp(T \widehat{C}_T \cC_{\infty}^2) \to 0 \quad \mbox{as $N \to \infty$, for any $T>0$}.
\end{align*}
The conclusion follows by Lemma \ref{lim-tau-lem}.
\end{proof}

\begin{lemma}
\label{local-sol}
If $\gamma > d$ and Assumption \ref{assumpt} holds, then $\tu_N$ is a local solution of \eqref{fSWE} up to time $\tau_N$, i.e. for any $t\geq 0$ and $x \in D$,
\begin{align}
\nonumber
1_{\{t<\tau_N\}} \tu_N(t,x)&=
1_{\{t<\tau_N\}} \int_0^t \int_{D}G_{t-s}(x,y) b\big(\tu_N(s,y)\big)dyds \\
\label{main-step1}
& \quad +\ 1_{\{t<\tau_N\}} \int_0^t \int_{D}
G_{t-s}(x,y)\sigma\big(\tu_N(s,y)\big) L(ds,dy) \quad \mbox{a.s.}
\end{align}
\end{lemma}

\begin{proof}
We fix $(t,x)$.
We write relation \eqref{def-uN} for $\tu_N$ and we multiply it by $1_{\{t<\tau_N\}}$. We obtain:
\begin{align*}
1_{\{t<\tau_N\}} \tu_N(t,x)& =1_{\{t<\tau_N\}} \int_0^t \int_{D}G_{t-s}(x,y)b_N\big(\tu_N(s,y)\big)dyds \\
& \quad + 1_{\{t<\tau_N\}}\int_0^t \int_{D}
G_{t-s}(x,y)\sigma_N\big(\tu_N(s,y)\big)L(ds,dy)\\
& =: \cA_N+\cB_N \quad \mbox{a.s.}
\end{align*}
We treat separately the two terms. For $\cA_N$, we have: 
\begin{align}
\nonumber
\cA_N &= 1_{\{t<\tau_N\}} \int_0^t \int_{D}G_{t-s}(x,y)b_N\big(\tu_N(s,y)\big)1_{\{s<\tau_N\}} dyds\\
\label{A1-1}
&=1_{\{t<\tau_N\}} \int_0^t \int_{D}G_{t-s}(x,y)b\big(\tu_N(s,y)\big)1_{\{s<\tau_N\}} dyds\\
\nonumber
&=1_{\{t<\tau_N\}} \int_0^t \int_{D}G_{t-s}(x,y)b\big(\tu_N(s,y)\big)dyds,
\end{align}
where we used the local property \eqref{local-Leb} of the Lebesgue integral for the first and last equality, and \eqref{bN-b} for the second equality. Similarly, for $\cB_N$, we use the local property of the stochastic integral with respect to $L$ (given by Lemma \ref{local-cor}) and relation \eqref{sN-s}:
\begin{align}
\nonumber
\cB_N &= 1_{\{t<\tau_N\}} \int_0^t \int_{D}G_{t-s}(x,y)\sigma_N\big(\tu_N(s,y)\big)1_{\{s<\tau_N\}} L(ds,dy)\\
\label{A2-1}
&=1_{\{t<\tau_N\}} \int_0^t \int_{D}G_{t-s}(x,y)\sigma\big(\tu_N(s,y)\big)1_{\{s<\tau_N\}} L(ds,dy)\\
\nonumber
&=1_{\{t<\tau_N\}} \int_0^t \int_{D}G_{t-s}(x,y)\sigma\big(\tu_N(s,y)\big)L(ds,dy).
\end{align}
This proves \eqref{main-step1}.

\end{proof}

\medskip

{\bf Proof of Theorem \ref{main}:} 

\medskip 

{\em Step 1.} In this step, we prove that the process $u$ defined by:
\[
u(t,x)=\sum_{N \geq 1}\tu_{N}(t,x)1_{[\tau_{N-1},\tau_N)}(t), \quad \mbox{where $\tau_0=0$},
\]
is a global solution of equation \eqref{fSWE}. By the consistency relation between $u_N$ and $u_{N+1}$ given by Lemma \ref{strong-uNN}, with probability 1, for any $N$,
\begin{equation}
\label{u-uN-id}
u(t,x)=\tu_N(t,x) \quad \mbox{for $t \in [0,\tau_N)$ and $x \in D$}.
\end{equation}
Let $\cA_N$ and $\cB_N$ be as in the proof of Lemma \ref{local-sol}. Recalling \eqref{A1-1} and \eqref{A2-1}, and using \eqref{u-uN-id}, followed by the local properties of the two integrals, we see that:
\begin{align*}
\cA_N&=1_{\{t<\tau_N\}} \int_0^t \int_{D}G_{t-s}(x,y)b\big(u(s,y)\big)dyds,\\
\cB_N&=1_{\{t<\tau_N\}} \int_0^t \int_{D}G_{t-s}(x,y)\sigma\big(u(s,y)\big)L(ds,dy).
\end{align*}
Therefore, relation \eqref{main-step1} becomes:
\begin{align*}
& 1_{\{t<\tau_N\}}u(t,x)\\
& \qquad =1_{\{t<\tau_N\}}\left(\int_0^t \int_{D}G_{t-s}(x,y)b\big(u(s,y)\big)dyds+ \int_0^t \int_{D}G_{t-s}(x,y)\sigma\big(u(s,y)\big)L(ds,dy)\right),
\end{align*}
a.s. We let $N \to \infty$. By Lemma \ref{lem:1}, $\lim_{N \to \infty}1_{\{t<\tau_N\}} =1$ a.s. Hence, for any $t\geq 0$ and $x \in D$,
\[
u(t,x)=\int_0^t \int_{D}G_{t-s}(x,y)b\big(u(s,y)\big)dyds+ \int_0^t \int_{D}G_{t-s}(x,y)\sigma\big(u(s,y)\big)L(ds,dy) \quad \mbox{a.s.}
\]

We now prove that $u$ satisfies \eqref{u-cadlag}. For this, we let
\[
\Omega^*=\bigcap_{N\geq 1}\left\{\tu_N(t,\cdot) \in H_r(D) \ \mbox{for any $t\geq 0$ and $t \mapsto \|\tu_N(t, \cdot) \|_{H_r(D)}$ is c\`adl\`ag}\right\}.
\]
Clearly, $\bP(\Omega^*)=1$. We show that both properties listed in \eqref{u-cadlag} hold on $\Omega^*$. Fix $\omega \in \Omega^*$.

(a) For any $t \geq 0$ arbitrary, there exists an integer $N\geq 1$ (depending on $(\omega,t)$) such that $t <\tau_{N}(\omega)$, and hence $u(\omega,t,\cdot)=\tu_{N}(\omega,t,\cdot)\in H_r(D)$. 

(b) To show that the map $t \mapsto \|u(\omega,t, \cdot) \|_{H_r(D)}$ is c\`adl\`ag at an arbitrary point $t_0$, we consider two cases. (i) If there exists $N$ such that $\tau_{N-1}(\omega)<t_0 <\tau_{N}(\omega)$, then the map $t \mapsto \|u(\omega,t, \cdot) \|_{H_r(D)}$ is c\`adl\`ag at $t_0$ because it coincides with the c\`adl\`ag map
$t \mapsto \|\tu_N(\omega,t, \cdot) \|_{H_r(D)}$ in a neighbourhood of $t_0$. (ii) If $t_0=\tau_N(\omega)$ for some $N\geq 1$, then the map $t \mapsto \|u(t, \cdot) \|_{H_r(D)}$ has a left limit at $t_0$ because it coincides with the c\`adl\`ag map
$t \mapsto \|\tu_N(\omega,t, \cdot) \|_{H_r(D)}$ in a left neighbourhood of $t_0$, and is right-continuous at $t_0$ because it coincides with the c\`adl\`ag map
$t \mapsto \|\tu_{N+1}(\omega,t, \cdot) \|_{H_r(D)}$ in a right neighbourhood of $t_0$.

\medskip

{\em Step 2.} In this step, we prove that $u$ is unique among all global solutions that satisfy \eqref{u-cadlag}. Let $v$ be another global solution of \eqref{fSWE} that satisfies \eqref{u-cadlag}. 
For each integer $N \geq 1$, consider the stopping time
\[ 
\tau_N^v := \inf\{ t > 0 ; \|v(t, \cdot)\|_{H_r(D)} > N \} \quad (\inf \emptyset=\infty).
\]
It is clear that $(\tau_N^v)_{N\ge1}$ is non-decreasing and that $\tau_N^v \to \infty$ 
a.s. as $N \to \infty$.
If $t <\tau_N^v$, then $\|v(t, \cdot)\|_{H_r(D)} \leq  N$, and $\|v(t, \cdot)\|_{L^{\infty}(D)}\leq \cC_{\infty}N$ (by Theorem \ref{embed-th}). So,
\begin{equation}
\label{bN-b-v} b_{N}(v(t,x))=b(v(t,x)) \ \mbox{and} \ \sigma_{N}(v(t,x))=\sigma(v(t,x)) \quad \mbox{if} \ t <\tau_N^v.
\end{equation} 
Hence, for any stopping time $\tau \leq \tau_N^v$:
\begin{align*}
     & v(t,x) 1_{\{t < \tau\}}   \\
     & =1_{\{t < \tau\}} \Bigg( \int_0^t \int_D G_{t-s}(x,y) b_N\big(v(s,y) \big) 1_{\{s < \tau\}}   dyds+ \int_0^t  \int_D G_{t-s}(x,y) \sigma_N \big(v(s,y)\big) 1_{\{s < \tau\}}   L(ds,dy) \Bigg)\\ 
   &=   1_{\{t < \tau\}} \Bigg( \int_0^t \int_D G_{t-s}(x,y) b_N\big(v(s,y) 1_{\{s < \tau\}} \big)  dyds  + \int_0^t  \int_D G_{t-s}(x,y) \sigma_N \big(v(s,y)1_{\{s < \tau\}} \big)  L(ds,dy) \Bigg),
\end{align*}
where for the first equality we used the local property of the stochastic integral, and for the second equality, we used the fact that for any function $f:\bR \to \bR$, 
\begin{align}
\nonumber
G_{t-s}(x,y)f\big(v(s,y)\big) 1_{\{s<\tau\}}& =G_{t-s}(x,y)f\big(v(s,y)1_{\{s<\tau\}}\big) 1_{\{s<\tau\}}\\
\label{1-inside}
&=G_{t-s}(x,y)f\big(v(s,y)1_{\{s<\tau\}}\big) \quad \mbox{on $\{t<\tau\}$}.
\end{align}
(For the last equality, $G_{t-s}(x,y)1_{\{s<\tau\}}=G_{t-s}(x,y)$ on  $\{t<\tau\}$, since $G_t=0$ if $t<0$). 
The same equality holds for the global solution $u$ constructed in {\em Step 1}, for any stopping time $\tau \leq \tau_N$. We use these two equalities for $\tau=\tau_N^v \wedge \tau_N=:T_N$. We obtain:
\begin{align*}
   &\Big( u(t,x) - v(t,x)  \Big) 1_{\{t < T_N\}} \\
   &= 1_{\{t < T_N\}}  \Bigg( \int_0^t \int_D G_{t-s}(x,y) \Big( b_N\big(u(s,y)  1_{\{s < T_N\}} \big) - b_N\big(v(s,y)  1_{\{s < T_N \}} \big) \Big) dyds  \\
     & \qquad\qquad + \int_0^t  \int_D G_{t-s}(x,y)  \Big( \sigma_N\big(u(s,y)  1_{\{s < T_N\}} \big) - \sigma_N\big(v(s,y)  1_{\{s < T_N\}}\big) \Big) L(ds,dy) \Bigg).
\end{align*}
Define
\[
g(t)=\sup_{x \in  D } \bE \Big[\big(u(t,x) - v(t,x) \big)^2 1_{\{t < T_N \}} \Big].
\]
Since $b_N$ and $\sigma_N$ are globally Lipschitz, 
arguing as in the proof of Lemma \ref{lem-uN-u}, we get:
\[
g(t) \le C_{T,N} \int_0^t g(s)J(t-s)ds \quad \mbox{for any $t \in [0,T]$},
\]
where $C_{T,N}$ and $J$ are defined by \eqref{def-CJ}.
By Lemma 15 of \cite{dalang99}, it follows that $g(t)=0$ for any $t \geq 0$. Hence, for any  $t \geq 0$ and $x \in D$,
\[
u(t,x) 1_{\{t < T_N\}}
= 
v(t,x) 1_{\{t < T_N\}}
\quad \text{a.s.}
\]
Letting $N \to \infty$ in the equality above, and using the fact that $\lim_{N \to \infty} T_N = \infty$ a.s., we conclude that for every 
$(t,x) \in \mathbb{R}_+ \times D$, one has $u(t,x) = v(t,x)$ a.s.
\qed


\section{The infinite variance case}
\label{section-inf}

In this section, we consider the case of equation \eqref{fSWE2} driven by a L\'evy basis $\Lambda$, where we assume that the Lévy measure $\nu$ is symmetric. We include some background material in Section \ref{subsect-Levy-bases}, and we present the proof of Theorem \ref{main2} in Section \ref{section-inf-var}.

\subsection{Integration with respect to L\'evy bases}
\label{subsect-Levy-bases}

In this section, we recall some basic material related to integration with respect to L\'evy bases. We refer the reader to \cite{chong17-JTP,chong17-SPA} for more details.

For any $A \in \cP_b$, let $\int 1_{A}d\Lambda:=\Lambda(A)$. By linearity, we extend this
 definition to the class $\cS$ of linear combinations of indicators $1_A$ with $A \in \cP_b$. For any predictable process $H$, we define the {\em Daniell mean}: 
\[
\|H\|_{\Lambda}=\sup_{S \in \cS,|S|\leq |H|}\left\|\int S d\Lambda \right\|_{L^0(\Omega)},
\]
where $L^0(\Omega)$ is the set of all random variables defined on $(\Omega,\cF,\bP)$ equipped with the pseudo-norm $\|X\|_{L^0(\Omega)}=\bE[|X|\wedge 1]$.

\begin{definition}
\label{def-integrable}
{\rm
A predictable process $H$ is {\em integrable} with respect to $\Lambda$ if there exists a sequence $(S_n)_{n\geq 1}$ in $\cS$ such that $\|S_n-H\|_{\Lambda}\to 0$ as $n \to \infty$. 
We denote by $L^{0}(\Lambda)$ the class of integrable processes with respect to $\Lambda$.
}
\end{definition}


For any $S \in \cS$, we denote $I^{\Lambda}(S)=\int S d\Lambda$. The map $I^{\Lambda}: \cS \to L^0(\Omega)$ is {\em not} an isometry. But the fact that this map satisfies the following trivial inequality
\[ \|I^{\Lambda}(S)\|_{L^0(\Omega)}\leq \|S\|_{\Lambda} \quad \mbox{for all $S \in \cS$}\]
is sufficient for extending $I^{\Lambda}$ from $\cS$ to $L^{0}(\Lambda)$. More precisely, if $H \in L^{0}(\Lambda)$ and $(S_n)_{n\geq 1}$ is the approximating sequence of simple integrands given by Definition \ref{def-integrable}, then 
$\{I^{\Lambda}(S_n)\}_{n\geq 1}$ is a Cauchy sequence in $L^0$ since
\begin{align*}
\|I^{\Lambda}(S_n)-I^{\Lambda}(S_m)\|_{L^0(\Omega)} \leq \|S_n-S_m\|_{\Lambda} \leq \|S_n-H\|_{\Lambda}+
\|S_m-H\|_{\Lambda} \to 0,
\end{align*}
as $n,m \to \infty$. By definition, we set
$I^{\Lambda}(H)=\lim_{n \to \infty} I^{\Lambda}(S_n)$ in $L^0(\Omega)$,
and we say that $I^{\Lambda}(H)$ is the {\em stochastic integral} of $H$ with respect to $\Lambda$. We use the notation:
\[
I^{\Lambda}(H)=\int_0^{\infty}\int_{\bR^d} H(t,x) \Lambda(dt,dx). 
\]

\subsection{Proof of Theorem \ref{main2}}
\label{section-inf-var}

In this section, we give the proof of Theorem \ref{main2}. As mentioned in the Introduction, we consider first the equation with truncated noise $\Lambda^K$ (which has a global solution $u^K$), and then we paste the solutions $(u^K)_{K\geq 1}$ to construct a global solution of equation \eqref{fSWE2}.

\medskip

The truncated noise $\Lambda^K$ is defined by:
\begin{equation}
\label{def-LK}
\Lambda^K(A)=\int_{0}^{\infty}\int_D \int_{\{|z|\leq 1\}}1_{A}(t,x)z \widetilde{J}(dt,dx,dz)+
\int_{0}^{\infty}\int_D \int_{\{1<|z|\leq K\}}1_{A}(t,x)z J(dt,dx,dz),
\end{equation}
for any $A \in \cP_b$, and the associated stopping time is:
\begin{equation}
\label{def-tauK}
\tau^K=\inf\{t\geq 0; J\big([0,t] \times D \times \{|z|>K\}\big)>0 \} \quad (\inf \emptyset=\infty).
\end{equation}
Clearly, $\tau^K  \leq \tau^{K+1}$ for any $K\geq 1$. Moreover, by Lemma \ref{lim-tau-lem}, $\lim_{K \to \infty}\tau^K =\infty$ a.s., since for any $T>0$,
\[
\bP(\tau^K >T) =\bP\big(J\big([0,T] \times D \times \{|z|>K\}\big)=0\big)=e^{-T|D| \nu(\{|z|>K\})} \to 1 \quad \mbox{as} \ K \to \infty.
\]
Note that, for any $t>0$ and $A \in \cP_b$ with $A \subseteq \Omega \times [0,t] \times D$,
\begin{equation}
\label{LLK}
\Lambda(A)=\Lambda^K(A) \quad \mbox{on the event $\{t <\tau^K\}$}.
\end{equation}
Next, we observe that, due to the {\em symmetry} of $\nu$, for any $B \in \cB_b(\bR_{+}\times \bR^d)$, 
\begin{align*}
\int_{B \times \{1<|z|\leq K\}}z \widetilde{J}(dt,dx,dz)&=
\int_{B \times \{1<|z|\leq K\}}z J(dt,dx,dz)-|B|\int_{\{1<|z|\leq K\}}z \nu(dz)\\
&=\int_{B \times \{1<|z|\leq K\}}z J(dt,dx,dz),
\end{align*}
and therefore, for any $B \in \cB_b(\bR_{+}\times \bR^d)$,
\[
L^K(B):=\Lambda^K(\Omega \times B)=\int_{B \times \{|z|\leq K\}}z \widetilde{J}(dt,dx,dz).
\]
Hence, the process $L^K=\{L^K(B),B \in \cB_b(\bR_{+}\times \bR^d)\}$ is a (finite variance) L\'evy white noise as in \eqref{def-L1}. More precisely, for any $B \in \cB_b(\bR_{+}\times \bR^d)$, $L^K(B)$ can be written as:
\[
L^K(B)=\int_{B \times \bR_0}z \widetilde{J}^K(dt,dx,dz),
\]
where $J^K$ is the restriction of $J$ to $\bR_+ \times \bR^d \times \{|z|\leq K\}$, and $\widetilde{J}^K$ is the compensated version of $J^K$. The L\'evy measure of $L^K$ is $\nu^K:=\nu(\cdot \cap \{|z|\leq K\})$, which satisfies:
\[
m_2^K:=\int_{\bR_0}z^2 \nu^K(dz)=\int_{\{0<|z|\leq K\}}z^2 \nu(dz)<\infty.
\]

To be able to import the results from the previous section, we use the following fact:
if a predictable process $H$ is It\^o integrable with respect to $L^K$, then $H$ is integrable with respect to $\Lambda^K$ (in the sense of Definition \ref{def-integrable}), and the two integrals coincide:
\begin{equation}
\label{int-coincide}
\int_0^t \int_{\bR^d} H(s,x)L^K(ds,dx)=\int_0^t \int_{\bR^d} H(s,x) \Lambda^{K}(ds,dx) \quad \mbox{a.s.}
\end{equation}

\medskip

We now return to our problem. By Theorem \ref{main}, we know that if $\gamma>d$ and Assumption \ref{assumpt} holds, then equation \eqref{fSWE} with noise $L$ replaced by $L^K$ has a unique global solution $u^K$.
This means that for any $t\geq 0$ and $x \in D$,
\[
u^K(t,x)=\int_0^t \int_{D}G_{t-s}(x,y)b\big(u^K(s,y)\big)dyds+\int_0^t \int_{D}G_{t-s}(x,y)
\sigma\big(u^K(s,y)\big)L^K(ds,dy)  \quad \mbox{a.s.}
\]
In view of \eqref{int-coincide}, this implies that for any $t\geq 0$ and $x \in D$,
\begin{equation}
\label{uK-sol}
u^K(t,x)=\int_0^t \int_{D}G_{t-s}(x,y)b\big(u^K(s,y)\big)dyds+\int_0^t \int_{D}G_{t-s}(x,y)
\sigma\big(u^K(s,y)\big)\Lambda^K(ds,dy) \quad \mbox{a.s.}
\end{equation}
We recall that the solution $u^K$ is given by:
\[
u^K(t,x)=\sum_{N\geq 1}\tu_N^K(t,x)1_{[\tau_{N-1}^K,\tau_N^K)}(t) \quad \mbox{with $\tau_0^K=0$},
\]
where $u_N^K$ is the solution of equation \eqref{fSWE-N} with noise $L$ replaced by $L^K$  (and truncated coefficients $b_N$ and $\sigma_N$), $\tu_N^K$ is the modification of $u_N^K$ given by Theorem \ref{tu-th}, and
\[
\tau_N^K=\inf\{t\geq 0;\|\tu_N^K(t,\cdot)\|_{H_r(D)}>N\} \quad (\inf \emptyset=\infty).
\]
By Lemma \ref{strong-uNN}, with probability 1, for any $K$ and $N$,
\begin{equation} 
\label{consist1}
u^K(t,x)=\tu_N^K(t,x) \quad \mbox{for all $t\in [0,\tau_N^K)$ and $x \in D$}.
\end{equation}

\medskip

\begin{lemma}
\label{uu-lem}
If $\gamma>d$, Assumption \ref{assumpt} holds, and $\nu$ is symmetric, 
then for any $K \geq 1$,
\[
\bP\big(u^K(t,x)=u^{K+1}(t,x) \quad \mbox{for all $t \in [0,\tau^K)$ and $x \in D$} \big)=1.
\]
\end{lemma}

\begin{proof}
We proceed as in the proof of Lemma \ref{strong-uNN}. To compare $u^K$ with $u^{K+1}$, we will use their respective approximations $\tu_{N}^K$ and $\tu_{N}^{K+1}$.

By Lemma \ref{lem-t-Tu}, with probability 1, $\tu_N^K=\widetilde{\cT}_N^K \tu_N^K$, where 
$\widetilde{\cT}_N^K$ is defined as $\widetilde{\cT}$ (see \eqref{def-TT}), but with $(b,\sigma,L)$ replaced by $(b_N,\sigma_N,L^K)$. Hence, by \eqref{consist1},  on the event $\{t<\tau_N^K\}$, 
\begin{align}
\nonumber
\cF_k[u^K(t,\cdot)]&=\cF_k[\tu_N^K(t,\cdot)]=\cF_k[(\widetilde{\cT}_N^K \tu_N^K)(t,\cdot)]\\
&=\frac{1}{\lambda_k^{\gamma/2}} \int_0^t \int_{D} \sin\big(\lambda_k^{\gamma/2}(t-s) \big)e_k(y) b_N \big( \tu_{N}^K(s,y)\big) dyds\\
\nonumber
&+\frac{1}{\lambda_k^{\gamma/2}} \sin(\lambda_k^{\gamma/2} t) \int_0^t \int_{D} \cos\big(\lambda_k^{\gamma/2}s \big)e_k(y) \sigma_N \big( \tu_{N}^K(s,y)\big) \Lambda^K(ds,dy)\\
\label{Fk-uK}
&-\frac{1}{\lambda_k^{\gamma/2}} \cos(\lambda_k^{\gamma/2} t) \int_0^t \int_{D} \sin\big(\lambda_k^{\gamma/2}s \big)e_k(y) \sigma_N \big( \tu_{N}^K(s,y)\big) \Lambda^K(ds,dy),
\end{align}
with the convention that the stochastic integrals above have c\`adl\`ag sample paths.

Similarly, on the event
$\{t<\tau_N^{K+1}\}$, 
\begin{align}
\nonumber
\cF_k[u^{K+1}(t,\cdot)]&=\frac{1}{\lambda_k^{\gamma/2}} \int_0^t \int_{D} \sin\big(\lambda_k^{\gamma/2}(t-s) \big)e_k(y) b_N \big( \tu_{N}^{K+1}(s,y)\big) dyds\\
\nonumber
&+\frac{1}{\lambda_k^{\gamma/2}} \sin(\lambda_k^{\gamma/2} t) \int_0^t \int_{D} \cos\big(\lambda_k^{\gamma/2}s \big)e_k(y) \sigma_N \big( \tu_{N}^{K+1}(s,y)\big) \Lambda^{K+1}(ds,dy)\\
\label{Fk-uK+1}
&-\frac{1}{\lambda_k^{\gamma/2}} \cos(\lambda_k^{\gamma/2} t) \int_0^t \int_{D} \sin\big(\lambda_k^{\gamma/2}s \big)e_k(y) \sigma_N \big( \tu_{N}^{K+1}(s,y)\big) \Lambda^{K+1}(ds,dy).
\end{align}
The idea is to transform the integral with respect to $\Lambda^{K+1}$ into an integral with respect to $\Lambda^{K}$, so that we can compare the two Fourier coefficients above.
Note that by \eqref{LLK}, for any $A \in \cP_b$ with $A \subset \Omega \times [0,t]\times \bR^d$,
\begin{equation}
\label{Lambda}
\Lambda^{K}(A)=\Lambda^{K+1}(A)=\Lambda(A) \quad \mbox{on the event $\{t<\tau^K\}$}.
\end{equation}
Since the stopping times $\tau_{N}^K$, $\tau_{N}^{K+1}$ and $\tau^K$ are not comparable to each other, we consider:
\[
\rho_N^K:=\tau_{N}^K \wedge \tau_{N}^{K+1}\wedge \tau^K.
\]
Using Lemma \ref{local-cor2} with $\Lambda=\Lambda^{K}$ and $\tau=\rho_{N}^{K}$, we obtain that:
\begin{align*}
1_{\{t<\rho_N^K\}}\int_0^t \int_{\bR^d}H(s,x)\Lambda^{K}(ds,dx)&=
1_{\{t<\rho_N^K\}}\int_0^t \int_{\bR^d}H(s,x) 1_{\{s<\rho_N^K \}} \Lambda^{K}(ds,dx),
\end{align*}
for a suitable predictable process $H$. Using this relation for the two stochastic integrals on the right hand-side of \eqref{Fk-uK}, we obtain that on the event $\{t<\rho_N^K\}$,
\begin{align*}
\cF_k[u^{K}(t,\cdot)]&=\frac{1}{\lambda_k^{\gamma/2}} \int_0^t \int_{D} \sin\big(\lambda_k^{\gamma/2}(t-s) \big)e_k(y) b_N \big( \tu_{N}^{K}(s,y)\big) 1_{\{s<\rho_N^K \}} dyds\\
&+\frac{1}{\lambda_k^{\gamma/2}} \sin(\lambda_k^{\gamma/2} t) \int_0^t \int_{D} \cos\big(\lambda_k^{\gamma/2}s \big)e_k(y) \sigma_N \big( \tu_{N}^{K}(s,y)\big) 1_{\{s<\rho_N^K \}} \Lambda^{K}(ds,dy)\\
&-\frac{1}{\lambda_k^{\gamma/2}} \cos(\lambda_k^{\gamma/2} t) \int_0^t \int_{D} \sin\big(\lambda_k^{\gamma/2}s \big)e_k(y) \sigma_N \big( \tu_{N}^{K}(s,y)\big) 1_{\{s<\rho_N^K \}} \Lambda^{K}(ds,dy).
\end{align*}
Using Lemma \ref{local-cor2} with $\Lambda=\Lambda^{K+1}$ and $\tau=\rho_{N}^K$, followed by \eqref{Lambda}, we obtain that:
\begin{align*}
1_{\{t<\rho_N^K\}}\int_0^t \int_{\bR^d}H(s,x)\Lambda^{K+1}(ds,dx)&=
1_{\{t<\rho_N^K\}}\int_0^t \int_{\bR^d}H(s,x) 1_{\{s<\rho_N^K \}} \Lambda^{K+1}(ds,dx)\\
&=
1_{\{t<\rho_N^K\}}\int_0^t \int_{\bR^d}H(s,x) 1_{\{s<\rho_N^K \}} \Lambda^{K}(ds,dx).
\end{align*}
We use this relation for the two stochastic integrals on the right hand-side of \eqref{Fk-uK+1}. It follows that on the event $\{t<\rho_N^K\}$,
\begin{align*}
\cF_k[u^{K+1}(t,\cdot)]&=\frac{1}{\lambda_k^{\gamma/2}} \int_0^t \int_{D} \sin\big(\lambda_k^{\gamma/2}(t-s) \big)e_k(y) b_N \big( \tu_{N}^{K+1}(s,y)\big) 1_{\{s<\rho_N^K \}} dyds\\
&+\frac{1}{\lambda_k^{\gamma/2}} \sin(\lambda_k^{\gamma/2} t) \int_0^t \int_{D} \cos\big(\lambda_k^{\gamma/2}s \big)e_k(y) \sigma_N \big( \tu_{N}^{K+1}(s,y)\big) 1_{\{s<\rho_N^K \}} \Lambda^{K}(ds,dy)\\
\label{Fk-uK}
&-\frac{1}{\lambda_k^{\gamma/2}} \cos(\lambda_k^{\gamma/2} t) \int_0^t \int_{D} \sin\big(\lambda_k^{\gamma/2}s \big)e_k(y) \sigma_N \big( \tu_{N}^{K+1}(s,y)\big) 1_{\{s<\rho_N^K \}} \Lambda^{K}(ds,dy).
\end{align*}
The two Fourier transform expressions now have the {\em same} integrator $\Lambda^{K}$, which means that we are in the position to use a similar argument as in the proof of Lemma \ref{strong-uNN}. We will use:
(i) Cauchy-Schwarz inequality for the Lebesgue integral, which produces the factor $T|D|$; 
(ii) Doob maximal inequality for each (c\`adl\`ag) stochastic integral with respect to the $\Lambda^{K}$, which produces the factor $4 m_2^{K}$; (iii) the Lipschitz properties of the functions $b_N$ and $\sigma_N$, which produce the factors $L_{N}^{(b)}$, respectively $L_{N}^{(\sigma)}$; (iv) the embedding of $H_r(D)$ into $L^{\infty}(D)$, which produces the factor $\cC_{\infty}$. To be consistent with all the estimates used above, we use the inequality $(a+b+c)^2 \leq 2 [a^2+2(b^2+c^2)]$ to separate the three terms. 

We obtain:
\begin{align*}
& \bE\Big[\sup_{t\leq T} \|u^K(t,\cdot)-u^{K+1}(t,\cdot)\|_{H_r(D)}^2 1_{\{t<\rho_N^K\}} \Big]\\
& \qquad \leq C_N^K \int_0^T \bE\big[\sup_{\ell \leq s}\|u^K(s,\cdot)-u^{K+1}(s,\cdot)\|_{H_r(D)}^2 1_{\{s<\rho_K^N\}}  \big] ds,
\end{align*}
where $C_N^K=2 \cC_{\infty}^2 \big(T|D| (L_N^{(b)})^2 + 16 m_2^K (L_N^{(\sigma)})^2 \big)\sum_{k\geq 1}\lambda_{k}^{r-\gamma}$ has the same form as $C_N$ given by \eqref{def-CN}, but with $m_2$ replaced by $m_2^K$. By Gronwall lemma,
\[
\bE\Big[\sup_{t\leq T} \|u^K(t,\cdot)-u^{K+1}(t,\cdot)\|_{H_r(D)}^2 1_{\{t<\rho_N^K\}} \Big]=0 \quad \mbox{for all} \quad T>0.
\]
From this, as in the proof of Lemma \ref{strong-uNN}, we infer that $P(A_N^{K})=1$ for all $N,K\geq 1$, where
\[
A_N^K=\big\{u^K(t,x)=u^{K+1}(t,x)\quad \mbox{for all $t \in [0,\rho_N^K)$ and $x\in D$} \big\}.
\]
Note that $A_{N+1}^K \subseteq A_N^K$, since $\rho_{N+1}^K \leq \rho_N^K$. Finally, since $\lim_{N \to \infty}\rho_N^K=\tau^K$, it follows that
\[
A^K:=\bigcap_{N\geq 1}A_N^K=\big\{u^K(t,x)=u^{K+1}(t,x)\quad \mbox{for all $t \in [0,\tau^K)$ and $x\in D$} \big\},
\] 
and hence, $\bP(A^K)=\lim_{N\to \infty}\bP(A_N^K)=1$ for all $K\geq 1$.

\end{proof}

\bigskip
{\bf Proof of Theorem \ref{main2}:} For any $t\geq 0$ and $x \in D$, we let
\[
u(t,x)=\sum_{K\geq 1}u^K(t,x)1_{[\tau^{K-1},\tau^K)}(t) \quad \mbox{with $\tau^0=0$}.
\]
By Lemma \ref{uu-lem}, 
\begin{equation}
\label{uuk}
u(t,x)=u^{K}(t,x) \quad \mbox{for all $t \in [0,\tau^K)$ and $x \in D$}.
\end{equation}
To show that $u$ is a global solution of equation \eqref{fSWE2}, we proceed as in of the proof of Theorem \ref{main}. Using \eqref{uK-sol}, we have: for any $t\geq 0$ and $x \in D$, with probability 1,
\begin{align}
\nonumber
1_{\{t<\tau^K\}}u(t,x)=1_{\{t<\tau^K\}}u^K(t,x)& =1_{\{t<\tau^K\}}\int_0^t \int_{D}G_{t-s}(x-y)b\big(u^K(s,y)\big)dyds\\
\nonumber
& \quad +1_{\{t<\tau^K\}}\int_0^t \int_{D} G_{t-s}(x-y) \sigma \big( u^K(s,y)\big) \Lambda^K(ds,dy)\\
\label{AK-BK}
&=:\cA^K +\cB^K \quad \mbox{a.s.}
\end{align}
We treat separately the two terms. For the Lebesgue integral, using \eqref{uuk}, we have:
\begin{align*}
\cA^K &=1_{\{t<\tau^K\}}\int_0^t \int_{D}G_{t-s}(x-y)b\big(u^K(s,y)\big)1_{\{s<\tau^K\}}dyds\\
&= 1_{\{t<\tau^K\}} \int_0^t \int_{D}G_{t-s}(x-y)b\big(u(s,y)\big)1_{\{s<\tau^K\}}dyds\\
&= 1_{\{t<\tau^K\}} \int_0^t \int_{D}G_{t-s}(x-y)b\big(u(s,y)\big)dyds.
\end{align*}
For the stochastic integral, applying Lemma \ref{local-cor2} (for $\Lambda^K$ and $\tau^K$), relations \eqref{LLK} and \eqref{uuk}, and again Lemma \ref{local-cor2} (this time for $\Lambda$ and $\tau^K$), we obtain:
\begin{align*}
\cB^K &=1_{\{t<\tau^K\}}\int_0^t \int_{D} G_{t-s}(x-y) \sigma \big( u^K(s,y)\big) 1_{\{s<\tau^K\}} \Lambda^K(ds,dy)\\
& = 1_{\{t<\tau^K\}}\int_0^t \int_{D} G_{t-s}(x-y) \sigma \big( u^K(s,y)\big) 1_{\{s<\tau^K\}} \Lambda(ds,dy)\\
&=1_{\{t<\tau^K\}}\int_0^t \int_{D} G_{t-s}(x-y) \sigma \big( u(s,y)\big) 1_{\{s<\tau^K\}} \Lambda(ds,dy)\\
&=1_{\{t<\tau^K\}}\int_0^t \int_{D} G_{t-s}(x-y) \sigma \big( u(s,y)\big) \Lambda(ds,dy).
\end{align*}
Letting $K\to \infty$ in \eqref{AK-BK} and using $\lim_{K\to \infty}\tau^K=\infty$ a.s. we obtain that $u$ is a global solution of \eqref{fSWE2}. The fact that $u$ satisfies \eqref{u-cadlag} follows as in the proof of Theorem \ref{main}.

\medskip

For the uniqueness, let $v$ be another solution of \eqref{fSWE2} which satisfies \eqref{u-cadlag}.
Let 
\[ 
\tau_N^v := \inf\{ t > 0 ; \|v(t, \cdot)\|_{H_r(D)} > N \} \quad (\inf \emptyset=\infty).
\]
Then, $\tau_N^v \to \infty$ as $N\to \infty$, and \eqref{bN-b-v} holds.
Using the local property of the stochastic integral, relations \eqref{bN-b-v} and \eqref{1-inside}, followed by relation \eqref{LLK} between $\Lambda$ and $\Lambda^K$, and relation \eqref{int-coincide} between $\Lambda^K$ and $L^K$, it follows that for any stopping time $\tau \leq \tau_N^v \wedge \tau^K$,
\begin{align*}
v(t,x)1_{\{t<\tau\}} & =1_{\{t<\tau\}} \int_0^t \int_{D} G_{t-s}(x,y) b_N\big(v(s,y)  1_{\{s<\tau \}} \big) dyds\\
& \qquad + 1_{\{t<\tau\}}  
\int_0^t \int_{D} G_{t-s}(x,y) \sigma_N\big(v(s,y)  1_{\{s<\tau \}} \big) L^K(ds,dy).
\end{align*}
The same equality is true for the solution $u$ constructed in {\em Step 1}, for any stopping time $\tau \leq \tau^K \wedge \tau_N^K$, because $u(t,x)=u^K(t,x)$ on $\{t<\tau\}$ (since $\tau \leq \tau^K$), and
$u^K(t,x)=\tu_N^K(t,x)$ on $\{t<\tau\}$ (since $\tau \leq \tau_N^K$). (Recall that $\tu_N^K$ is a solution of equation \eqref{fSWE-N} with $L$ replaced by $L^K$.) 

We will use these two equalities for $\tau=\tau_N^v \wedge \tau^K \wedge \tau_N^K=:T_N^K$. Since $b_N$ and $\sigma_N$ are globally Lipschitz, we can now use the same argument as in the \underline{uniqueness} part of Theorem \ref{main} (with $L$ replaced by $L^K$), to conclude that for any $(t,x) \in \mathbb{R}_+ \times D$,
\[
u(t,x) 1_{\{t < T_N^K\}}= v(t,x) 1_{\{t < T_N^K\}}
\quad \text{a.s.}
\]
for each $K,N \in \mathbb{N}$. Letting $N \to \infty$, and using the fact that $\lim_{N \to \infty}T_N^K=\tau^K$, we get:
\[
u(t,x) 1_{\{t < \tau^K\}}= v(t,x) 1_{\{t < \tau^K\}}
\quad \text{a.s.}
\]
for each $K \in \mathbb{N}$. Finally, letting $K \to \infty$, and using the fact that $\lim_{K \to \infty}\tau^K=\infty$, we infer that
$u(t,x) = v(t,x)$ a.s. for all $(t,x) \in \mathbb{R}_+ \times D$. 
\qed

\bigskip

{\em Acknowledgement.}(i) R.B. acknowledges funding from the Office of Vice-President, Research at University of Ottawa, through the Visiting Researchers Program-Europe, which was used for the visit of L. Q.-S. at University of Ottawa in March 2025, when this project was initiated. 
(ii) The authors would like to thank Eul\`alia Nualart for providing them with an updated version of preprint \cite{FKN24}. (iii) The authors are grateful to Carlo Marinelli for pointing out that in general, relation \eqref{L2-rep} may not hold pointwise, which had led to a mistake in Theorem \ref{tu-th} in an earlier version of the manuscript. 

\appendix

\section{Sobolev spaces}
\label{app-Sob}

In this section, we provide some background material about Sobolev spaces. We recall that $(\lambda_k)_{k\geq 1}$ denote the eigenvalues of $-\Delta$ with Dirichlet boundary conditions, and $(e_k)_{k\geq 1}$ are the corresponding eigenfunctions, which are smooth functions and form a complete orthonormal basis of $L^2(D)$.


For any $f \in L^2(D)$, the following equality holds in $L^2(D)$:
\begin{equation}
\label{L2-rep}
f=\sum_{k\geq 1}\cF_k[f] e_k,
\end{equation}
where $\cF_k[f]:=\langle f, e_k\rangle_{L^2(D)}$ is the $k$-th Fourier coefficient of $f$. Corollary \ref{point-rep} below shows that in some instances, equality \eqref{L2-rep} holds also pointwise.

\medskip


Let $r \in \bR$ be arbitrary. 
Let $E_0$ be the set of functions of the form $f=\sum_{k=1}^N a_k e_k$, and define
\[
\|f\|_{H_r(D)}:=\left(\sum_{k=1}^N \lambda_k^r a_k^2\right)^{1/2}.
\]
As in Section 2 of \cite{CD23}, we define the {\em fractional Sobolev space} $H_r(D)$ of order $r$ as the completion of $E_0$ with respect to $\|\cdot\|_{H_r(D)}$. Each element $\Phi$ of $H_r(D)$ can be identified with a series of the form
\[
\Phi=\sum_{k\geq 1}
a_k(\Phi)e_k \quad
\mbox{with} \quad a_k(\Phi) \in \bR \quad \mbox{and} \quad  \|\Phi \|_{H_r(D)} := \left(\sum_{k \geq 1}\lambda_k^r |a_k(\Phi)|^2\right)^{1/2}<\infty.
\]
$H_r(D)$ is a Hilbert space with the inner product $\langle \cdot, \cdot \rangle_{H_r(D)}$ given by
\[
\langle \Phi_1, \Phi_2\rangle_{H_r(D)} := \sum_{k=1}^{\infty} \lambda_k^r a_k(\Phi_1)a_k(\Phi_2), \quad \Phi_1, \Phi_2\in H_r(D).
 \]
Moreover, $H_s(D) \subseteq H_r(D)$ for $s \leq r$ and $H_0=L^2(D)$. The evaluation
\[
\langle \Phi_1, \Phi_2 \rangle := \sum_{k =1}^{\infty} a_k (\Phi_1) a_k ( \Phi_2 ), \quad \Phi_1 \in H_{-r}(D), \, \Phi_2 \in H_r(D),
\]
puts $H_r(D)$ and $H_{-r}(D)$ in duality. If $f \in H_r(D)$, then  
\begin{center}
$a_k(f)=\cF_k[f]$ if $r\geq 0$, and $a_k(f)=\langle f,e_k \rangle$ if $r<0$.
\end{center}
When $r<0$, $\langle f,e_k \rangle$ denotes the action of the distribution $f$ on the test function $e_k$.

The spectral power of $-\Delta$ of order $\gamma >0$ satisfies:
\[
(-\Delta)^{\gamma}:\bigcup_{r \in \bR}H_r(D) \to \bigcup_{r\in \bR}H_r(D), \quad (-\Delta)^{\gamma}\Phi:=\sum_{k\geq 1} \lambda_k^{\gamma}a_k(\Phi) e_k.
\]

For normed vector spaces $X$ and $Y$, we write $X \hookrightarrow Y$ if $X$ is continuously embedded in $Y$, i.e. $X \subset Y$ and there exists $C>0$ such that $\|x\|_{Y} \leq C \|x\|_{X}$ for all $x \in X$.

\medskip

The following result plays a crucial role in the present article.

\begin{theorem}
\label{embed-th}
If $D$ is bounded and $\partial D$ is of class $C^{\infty}$, then
\[
H_r(D) \hookrightarrow L^{\infty}(D) \quad \mbox{for any} \ r>\frac{d}{2},
\]
i.e. $H_r(D) \subset L^{\infty}(D)$ and there exists a constant $\cC_{\infty}>0$ such that
\begin{equation}
\label{embed-eq}
\|h\|_{L^{\infty}(D)} \leq \cC_{\infty} \|h\|_{H_r(D)}, \quad \mbox{for all} \ h \in H_r(D).
\end{equation}
\end{theorem}

\begin{proof} 
From the last part of the proof of Lemma 2.18 of \cite{CDH19}, we know that 
\[
H_r(D) \hookrightarrow H^r(D) \quad \mbox{for any $r \geq 0$},
\] where $H^r(D)$ is defined by (9.1), Chapter 1 of \cite{lion-magenes68}.
By Theorem 9.8, Chapter 1 of \cite{lion-magenes68},  
\[
H^r(D)  \hookrightarrow C_0(\overline{D})  \quad \mbox{for any $r >\frac{d}{2}$},
\]
provided that $D$ is bounded and $\partial D$ is of class $C^{\infty}$, where $C_0(\overline{D})$ is the space of continuous functions on $\overline{D}$ equipped with the $\|\cdot\|_{L^{\infty}(D)}$-norm. 
\end{proof}

\medskip

\begin{corollary}
\label{point-rep}
Assume that $D$ is bounded and $\partial D$ is of class $C^{\infty}$. If $f \in H_r(D)$ for some $r>d/2$, then 
\[
f(x)=\sum_{k\geq 1}\cF_k[f] e_k(x) \quad \mbox{for all $x \in D$}.
\]
\end{corollary}

\begin{proof} We define $h_n(x)=f(x)-\sum_{k=1}^n \cF_k[f]e_k(x)$. Then
\[
\|h_n\|_{H_r(D)}^2=\sum_{k\geq n+1}\lambda_k^{r} \big(\cF_k[f]\big)^2 \to 0 \quad \mbox{as $n \to \infty$}.
\]
By Theorem \ref{embed-th}, $\|h_n\|_{L^{\infty}(D)} \to 0$ as $n \to \infty$. Hence, $\lim_{n \to \infty}h_n(x)=0$ for all $x \in D$.
\end{proof}

\section{Auxiliary Results}

In this appendix section, we give some elementary results which were used in the paper.

\medskip



\begin{lemma}
\label{fact1-lem}
For any random variable $X$ and set $A \in \cF$,
\begin{equation}
\label{fact1}
\mbox{$1_A X=0$  a.s.} \quad \mbox{if and only if} \quad  \mbox{$X=0$ a.s. on $A$}.
\end{equation}
Here ``$X=0$ a.s. on $A$'' means that $\bP\big(\{X=0\} \cap A\big)=\bP(A)$.
\end{lemma}

\begin{proof}
$1_A X=0$ a.s. means that $\bP(\Omega_0)=1$, where $\Omega_0=\{1_A X=0\}=\big(\{X=0\} \cap A\big)\cup A^c$. Since 
\[
\bP(\Omega_0)=\bP\big(\{X=0\} \cap A\big)+\bP(A^c),
\]
the fact that $\bP(\Omega_0)=1$ is equivalent to $\bP\big( \{X=0\} \cap A\big)=1-\bP(A^c)=\bP(A)$. 
\end{proof}


\begin{lemma}
\label{stop}
Let $B \subset \bR$ be an open set, and $(X_t)_{t\geq 0}$ be a real-valued right-continuous process, which is adapted with respect to a right-continuous filtration $(\cF_t)_{t\geq 0}$. Let
$$\tau := \inf \{ t \geq 0 ; X(t) \in B \}.$$
Then $\tau$ is a stopping time with respect to $(\cF_t)_{t\geq 0}$.
\end{lemma}

\begin{proof} This argument is given in the proof of Theorem 3 (Chapter I) of \cite{protter}, where it is stated that $X$ is c\`adl\`ag. We include the proof which shows that the fact that $X$ has left-limits is not needed.
Since $(\cF_t)_{t\geq 0}$ is right-continuous, it suffices to prove that 
\[
\{\tau<t\}\in \cF_t \quad \mbox{for all $t> 0$}.
\] 
To see this, let $r>t$ be arbitrary. For any $u \in (t,r)$, $\{\tau<u\} \in \cF_u \subset \cF_r$. Hence, $\{\tau \leq t \} =\cap_{u \in (t,r) \cap \bQ} \{\tau<u\} \in \cF_r$. Since $r>t$ was arbitrary, $\{\tau \leq t\}\in \cap _{r>t}\cF_r=\cF_{t+}=\cF_t$.

We claim that
\[
\{\tau<t\}=\bigcap_{s \in \bQ \cap [0,t)}\{X(s) \in B\}.
\]
To see this, let $S=\{ t > 0 : X(t) \in B \}$. If $\tau<t$, there exists $t_0\in S$ such that $t_0<t$. Hence $X(t_0)\in B$. But $X(t_0)=\lim_{s\downarrow t_0,s \in \bQ}X(s)$, since $X$ is right-continuous. So, there exists $s \in [t_0,t) \cap \bQ$ such that $X(s) \in B$. For the reverse inclusion, if $X(s) \in B$ for some $s \in \bQ \cap [0,t)$, then $s \in S$. Since $\tau=\inf S$, we must have $\tau \leq s$. So, $\tau<t$.
\end{proof}

\begin{lemma}
\label{cont-g-lem}
Let $g:[0,\infty) \to [0,\infty)$ be a right-continuous function and $a>0$. Define
$$\tau=\inf\{t>0; g(t)>a\}, \quad \mbox{with} \ \inf \emptyset=\infty.$$ 
a) If $\tau<\infty$, then $g(\tau)\geq a$. 
b) If $\tau \leq T$, then $\sup_{t \leq T} g(t) \geq a$.
\end{lemma}

\begin{proof}
a) Let $A=\{t>0; g(t)>a\}$. By the definition of the infimum, for any $n\geq 1$, there exists $t_{n} \in A$ such that $\tau \leq t_{n}<\tau+\frac{1}{n}$. Hence $g(t_n)>a$ for all $n$, and $t_n \to \tau$. Since $g$ is right-continuous, $g(t_n) \to g(\tau)$. Hence, $g(\tau) \geq g(a)$. 

b) If $g(\tau)> a$, then $\sup_{t \leq T}g(t) \geq g(\tau)>a$. If $g(\tau)=a$, we have two cases:\\
(i) If there exists $t^* \in (\tau,T]$ such that $g(t^*)>a$, then
$\sup_{t \leq T} g(t) \geq \sup_{t \in (\tau, T]} g(t)>a$.\\
(ii) Otherwise, $g(t)\leq a$ for all $t \in (\tau,T]$. Since 
$g(t)\leq a$ for all $t \in [0,\tau)$ (by the definition of $\tau$), it follows that 
$\sup_{t \leq T} g(t) =g(\tau)=a$.
\end{proof}


\begin{lemma}
\label{lim-tau-lem},
Let $(\tau_N)_{N\geq 1}$ be a sequence of stopping times such that $\tau_N \leq \tau_{N+1}$ a.s. for all $N \geq 1$. If 
 $\lim_{N \to \infty}\bP(\tau_N \leq T)=0$ for any $T>0$, then
$\lim_{N \to \infty}\tau_N=\infty$ a.s.
\end{lemma}

\begin{proof}
Let $\Omega_0=\{\tau_{N}\leq \tau_{N+1}\ \mbox{for all} \ N\geq 1\}$. Then $\bP(\Omega_0)=1$.
Let $A_N^{(T)}=\{\tau_N \leq T\}\cap \Omega_0$. Then $A_{N+1}^{(T)} \subset A_N^{(T)}$ for any $N \geq 1$, and hence, 
\[
\bP\big(\bigcap_{N\geq 1}A_N^{(T)}\big)=\lim_{N \to \infty}\bP( A_N^{(T)})=0 \quad \mbox{for any} \ T>0.
\]
This shows that $\bP\big(\{\tau_N \leq T \ \mbox{for all} \ N \geq 1\} \cap \Omega_0\big)=0$ for any $T>0$. Then $\bP(A)=0$, where
\[
A=\Omega_0 \cap \bigcup_{T \in \bQ_{+}} \{ \tau_N \leq T \ \mbox{for all} \ N \geq 1 \}. 
\]
Finally, the conclusion follows, once we observe that 
\[
A^c=\Omega_0^c \cup \bigcap_{T \in \bQ_{+}} \{ \tau_N > T \ \mbox{for some} \ N \geq 1 \}= \Omega_0^c \cup \{\lim_{N \to \infty} \tau_N=\infty\}.
\]
\end{proof}

\section{Local property of the stochastic integral}
\label{app-local}

In this section, we include some local property of the stochastic integrals with respect to $L$ and $\Lambda$. Lemma \ref{local-cor} was used in the proofs of Lemmas \ref{lem-uN-u}, \ref{strong-uNN} and \ref{local-sol}, while Lemma \ref{local-cor2} was used in the proof of Lemma \ref{uu-lem} and in the proof of Theorem \ref{main2}. 
For these results, we refer to relation (2.22) of \cite{BJ83}.




\begin{lemma}
\label{local-cor}
Let $L$ be the L\'evy white noise given by \eqref{def-L1}, and $H=\{H(t,x);t\geq 0,x\in \bR^d\}$ be a predictable process such that
\[
\bE \int_0^T \int_{\bR^d}|H(t,x)|^2 dtdx<\infty \quad \mbox{for all $T>0$}.
\]
Then for any stopping time $\tau$ and for any $t>0$,
\[
1_{[0,\tau)}(t)\int_0^t \int_{\bR^d} H(s,x)L(ds,dx)=1_{[0,\tau)}(t)\int_0^t \int_{\bR^d}H(s,x)1_{[0,\tau)}(s)L(ds,dx).
\]
\end{lemma}


\begin{lemma}
\label{local-cor2}
Let $\Lambda$ be an arbitrary L\'evy basis. Let $H=\{H(t,x);t\geq 0,x\in \bR^d\}$ be a predictable process for which there exists a sequence $(T_n)_{n\geq 1}$ of stopping times with $T_n \leq T_{n+1}$ and $\lim_{n \to \infty}T_n=\infty$, such that for any $n\geq 1$,
\[
\sup_{t\in [0,T]}\sup_{x \in D} \bE[H^2(t,x)1_{\{t\leq T_n\}}]<\infty \quad \mbox{for all $T>0$}.
\]
Then $H$ is integrable with respect to $\Lambda$, and for any stopping time $\tau$ and for any $t>0$,
\[
1_{[0,\tau)}(t)\int_0^t \int_{\bR^d} H(s,x)\Lambda(ds,dx)=1_{[0,\tau)}(t)\int_0^t \int_{\bR^d}H(s,x)1_{[0,\tau)}(s)\Lambda(ds,dx).
\]
\end{lemma}

\end{document}